\documentclass{amsart}

\newcommand{\supt}{\operatorname{supt}}
\newcommand{\dest}{\operatorname{dest}}
\newcommand{\src}{\operatorname{src}}

\usepackage{amssymb}
\usepackage[all]{xy}

\newtheorem{thm}{Theorem}[section]
\newtheorem{claim}[thm]{Claim}

\newtheorem*{ping_pong_lem}{Ping Pong Lemma}

\newtheorem{lemma}[thm]{Lemma}
\newtheorem{prop}[thm]{Proposition}
{\theoremstyle{definition} }
{\theoremstyle{definition} \newtheorem{example}[thm]{Example}}
\newtheorem{cor}[thm]{Corollary}

\renewcommand{\text}{\mathrm}

\newtheorem{innercustomthm}{Theorem}
\newenvironment{customthm}[1]
  {\renewcommand\theinnercustomthm{#1}\innercustomthm}
  {\endinnercustomthm}
  
\newcommand{\Bfrak}{\mathfrak{B}}

\newcommand{\R}{{\mathbf R}}

\newcommand{\Z}{{\mathbf Z}}

\newcommand{\PSL}{\operatorname{PSL}}

\newcommand{\HomeoI}{{\operatorname{Homeo}_+(I)}}
\newcommand{\HomeoR}{{\operatorname{Homeo}_+(\R)}}
\newcommand{\PLoI}{{\operatorname{PL}_+(I)}}

\newcommand{\gen}[1]{\langle #1 \rangle}

\theoremstyle{remark}
\newtheorem{remark}[thm]{Remark}

\newcommand{\str}[1]{\mathtt{#1}}
\newcommand{\Th}{{}^{\textrm{th}}}

\newcommand{\dvarbump}[3]
{
\xy
(0,0); (#3,#3)**\crv{(#1,#2)&(#2,#3)};
\endxy
}
\newcommand{\dvarnbump}[3]
{
\xy
(0,0); (#3,#3)**\crv{(#2,#1)&(#3,#2)};
\endxy
}

\makeindex

\begin{document}
\bibliographystyle{amsplain}

\title[Groups of fast homeomorphisms]{
Groups of fast homeomorphisms of the interval and the ping-pong argument}
\thanks{
This paper was prepared in part during a visit of the first and fourth author to
the \emph{Mathematisches Forschungsinstitut Oberwolfach}, Germany in December 2016
as part of their \emph{Research In Pairs} program.
The research of the fourth author was supported in part by
NSF grant DMS-1262019.}



\author[Bleak]{Collin~Bleak}
\author[Brin]{Matthew~G.~Brin}
\author[Kassabov]{Martin~Kassabov}
\author[Moore]{Justin~Tatch~Moore}
\author[Zaremsky]{Matthew~C.~B.~Zaremsky}

\address{
Collin~Bleak \\
School of Mathematics and Statistics \\
University of St. Andrews \\
St. Andrews, Fife KY16 9SS
}

\address{
Matthew~G.~Brin \\
Department of Mathematical Sciences \\
Binghamton University \\
Binghamton, NY 13902-6000
}

\address{
Martin~Kassabov, Justin~Tatch~Moore, Matthew~C.~B.~Zaremsky \\
Department of Mathematics \\
Cornell University \\
Ithaca, NY 14853-4201
}

\begin{abstract}
We adapt the Ping-Pong Lemma, which historically was used to study free
products of groups, to the setting of the homeomorphism group of the unit interval.
As a consequence, we isolate a large class 
of generating sets for subgroups of \(\HomeoI\) 
for which certain finite dynamical data can be used to determine
the marked isomorphism type of the groups which they generate.
As a corollary, we will obtain a criteria for embedding subgroups of 
\(\HomeoI\) into Richard Thompson's group \(F\).
In particular, every member of our class of generating sets
generates a group which embeds into \(F\) and in particular is not a free product.
An analogous abstract theory is also developed for groups of permutations of an infinite set.
 \end{abstract}

\maketitle

\section{Introduction}\label{IntroSec}

The \emph{ping-pong argument} was first used in
\cite[\S III,16]{Klein:Riemann} and \cite[\S II,3.8]{Fricke+Klein} to
analyze the actions of certain groups of linear fractional
transformations on the Riemann sphere.
Later distillations and generalizations of the arguments
(e.g., \cite[Theorem 1]{MR0148731})
were used to establish that a given group is a free product.
In the current paper we will adapt the \emph{ping-pong argument}
to the setting of subgroups of \(\HomeoI\), the
group of the orientation preserving homeomorphisms of the unit interval.
Our main motivation is to develop a better understanding of the finitely generated
subgroups of the group \(\PLoI\) of piecewise linear order-preserving homeomorphisms of
the unit interval.
The analysis in the current paper resembles the
original \emph{ping-pong argument} in that it establishes a tree structure on
certain orbits of a group action.
However, the arguments in the current paper differ from the usual analysis
as the generators of the group action have large sets of fixed points.

The focus of our attention in this article will be subgroups of \(\HomeoI\)
which are specified by what we will term \emph{geometrically fast} generating sets.
On one hand, our main result shows that the isomorphism types of the groups specified by
geometrically fast generating sets are determined by their \emph{dynamical diagram} 
which encodes their qualitative dynamics;
see Theorem \ref{CombToIso} below.
This allows us to show, for instance, that such sets generate groups which
are always embeddable into Richard Thompson's group \(F\).
On the other hand, we will see that a broad class of subgroups
of \(\PLoI\) can be generated using such sets.
This is substantiated in part by Theorem \ref{GeoProperGen} below.

At this point it is informative to consider an example.
First recall the classical \emph{Ping-Pong Lemma}
(see \cite[Prop. 1.1]{free_subgrp_linear}):

\begin{ping_pong_lem}
Let \(S\) be a set and \(A\) be a set of permutations of \(S\)
such that \(a^{-1} \not \in A\) for all \(a \in A\).
Suppose there is an assignment \(a \mapsto D_a \subseteq S\) of pairwise disjoint sets
to each \(a \in A^\pm := A \cup A^{-1}\) and an
\(x \in S \setminus \bigcup_{a \in A^\pm} D_a\) such that
if \(a \ne b^{-1}\) are in \(A^\pm\), then
\[
\big(D_b \cup \{x\}\big)a \subseteq D_a.
\] 
Then \(A\) freely generates \(\gen{A}\).
\end{ping_pong_lem}

\noindent
(We adopt the convention of writing permutations to the right of their arguments;
other notational conventions and terminology will be reviewed in Sections \ref{DefSec} and \ref{FastBumpsSec}.)
In the current paper, we relax the hypothesis so that the
containment \(D_b a \subseteq D_a\)
is required only when \(D_b\) is contained in the \emph{support}
of \(a\); similarly \(x a \in D_a\) is only required when
\(x a \ne x\).

Consider the three functions \((b_i \mid i < 3)\) in
\(\HomeoI\) whose graphs are shown in Figure \ref{FastF3}.
\begin{figure}
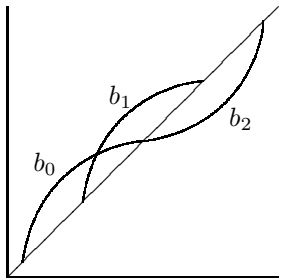

\[
\xy
(0,36); (0,0)**@{-}; (36,0)**@{-}; (0,0); (36,36)**@{-};
(10,9.6)*{\dvarbump{2}{14}{16}}; (5,15)*{b_0}; 
(26,25.6)*{\dvarnbump{2}{14}{16}}; (31,21)*{b_2};
(18,17.6)*{\dvarbump{2}{14}{16}}; (15,24)*{b_1};
\endxy
\]
\caption{Three homeomorphisms}\label{FastF3}
\end{figure}
\newcommand{\varbump}[3]
{
\xy
(0,0); (#3,0)**\crv{(#1,#1)&(#2,#1)};
\endxy
}
\newcommand{\varnbump}[3]
{
\xy
(0,0); (#3,0)**\crv{(#1,-#1)&(#2,-#1)};
\endxy
}
A schematic diagram (think of the line \(y=x\) as drawn
horizontally) of these functions might be:
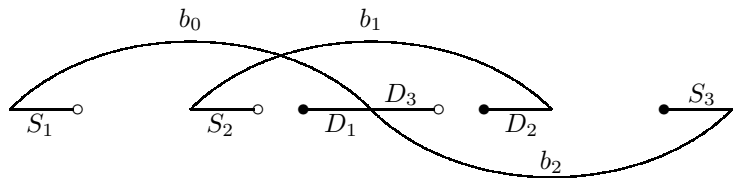
\[
\xy
(12,2.6)*{\varbump{12}{36}{48}};  (12,9)*{b_0};
(36,2.6)*{\varbump{12}{36}{48}};  (36,9)*{b_1};
(60,-9)*{\varnbump{12}{36}{48}}; (60,-10)*{b_2};
(-12,-3); (-3.5,-3)**@{-};(-3,-3)*{\circ}; (-8,-5)*{S_1}; (-3,-5)*{};
(12,-3); (20.5,-3)**@{-};(21,-3)*{\circ};  (16,-5)*{S_2}; (21,-5)*{};
(36,-3); (27.5,-3)**@{-};(27,-3)*{\bullet};  (32,-5)*{D_1}; (27,-5)*{};
(36,-3); (44.5,-3)**@{-};(45,-3)*{\circ};  (40, -1)*{D_3}; (45,-5)*{};
(60,-3); (51.5,-3)**@{-};(51,-3)*{\bullet};  (56, -5)*{D_2}; (51,-5)*{};
(84,-3); (75.5,-3)**@{-};(75,-3)*{\bullet};  (80,-1)*{S_3};  (75,-5)*{};
\endxy
\]
In this diagram, we have assigned intervals \(S_i\) and \(D_i\)
to the ends of the support to each \(b_i\) so that
the entire collection of intervals is pairwise disjoint. 
Our system of homeomorphisms is assumed to have an additional
dynamical property reminiscent of the hypothesis of the Ping-Pong Lemma:
\begin{itemize}

\item \(S_i b_i \cap D_i = \emptyset\);

\item \(b_i\) carries \(\supt(b_i) \setminus S_i\) into \(D_i\);

\item  \(b_i^{-1}\) carries \(\supt(b_i) \setminus D_i\) into \(S_i\) for each \(i\).

\end{itemize}

A special case of our main result
is that these dynamical requirements on the
\(b_i\)'s are sufficient to characterize the isomorphism type of
the group \(\gen{b_i \mid i < 3}\):
any triple \((c_i \mid i < 3)\)
which produces this same \emph{dynamical diagram} and satisfies these dynamical
requirements will generate a group isomorphic to \(\gen{b_i \mid  i < 3}\).
In fact the map \(b_i \mapsto c_i\) will extend to an isomorphism.
In particular,
\[
\gen{b_i \mid  i < 3} \cong
\gen{b_i^{k_i} \mid  i < 3}
\]
for any choice of \(k_i\geq 1\) for each \(i < 3\).

We will now return to the general discussion and be more precise.
Recall that if \(f\) is in \(\HomeoI\), then its \emph{support} is defined to
be \(\supt(f):=\{ t \in I \mid tf \ne t\}\).
A left (right) transition point of \(f\) is a \(t \in I \setminus \supt(f)\) such that for every \(\epsilon > 0\),
\((t,t+\epsilon) \cap \supt(f) \ne \emptyset\)
(respectively \((t-\epsilon,t) \cap \supt(f) \ne \emptyset\)).  
An \emph{orbital} of \(f\) is a component its support.
An orbital of \(f\) is \emph{positive} if \(f\) moves elements of the orbital to
the right;
otherwise it is negative.
If \(f\) has only finitely many orbitals, then the left (right) transition points
of \(f\) are precisely the left (right) end points of its orbitals.

A precursor to the notion of a \emph{geometrically fast} generating set is that of a
\emph{geometrically proper} generating set.
A set \(X \subseteq \HomeoI\) is
\emph{geometrically proper} if there is no element of \(I\) which is
a left transition point of more than one element of \(X\)
or a right transition point of more than one element of \(X\).
Observe that any geometrically proper generating set with only finitely many transition points
is itself finite.
Furthermore, geometrically proper sets are equipped with a canonical ordering
induced by the usual ordering on the least transition points of its elements.
While the precise definition of \emph{geometrically fast}
will be postponed until Section \ref{FastBumpsSec},
the following statements describe the key features of the definition:
\begin{itemize}

\item geometrically fast generating sets are geometrically proper.

\item if \(\{a_i \mid  i < n\}\) is geometrically proper, then
there is a \(k \geq 1\) such that 
\(\{a_i^k \mid  i < n\}\) is geometrically fast.

\item if \(\{a_i^k \mid  i < n\}\) is geometrically fast and \(k \leq k_i\) for \(i < n\), then
\(\{a_i^{k_i} \mid  i < n\}\) is geometrically fast.

\end{itemize}
\noindent
Our main result is that the isomorphism types of groups with geometrically fast generating sets
are determined by their qualitative dynamics.
Specifically, we will associate a \emph{dynamical diagram} to each geometrically fast set
\(\{a_i \mid i < n\} \subseteq \HomeoI\) which has finitely many transition points.
Roughly speaking, this is a record of the relative order of the orbitals
and transition points of the various \(a_i\), as well as the orientation of their orbitals.
In the following theorem \(M_X\) is a certain finite set of points chosen from the orbitals
of elements of \(X\) and \(M \gen{X} = \{t g \mid  t \in M_X \textrm{ and } g \in \gen{X}\}\).
These points will be chosen such that any nonidentity element of \(\gen{X}\) moves a point in \(M\gen{X}\).

\begin{thm} \label{CombToIso}
If two geometrically fast sets \(X , Y \subseteq \HomeoI\) have only finitely many transition points
and have isomorphic dynamical diagrams, then the induced
bijection between \(X\) and \(Y\) extends to an isomorphism of
\(\gen{X}\) and \(\gen{Y}\) (i.e. \(\gen{X}\) is \emph{marked isomorphic} to \(\gen{Y}\)).
Moreover, there is an order preserving bijection
\(\theta :  M \gen{X} \to M \gen{Y}\) such that
\(f \mapsto f^\theta\) induces the isomorphism 
\(\gen{X} \cong \gen{Y}\).
\end{thm}

We will also establish that under some circumstances the map \(\theta\) can be extended to a
continuous order preserving surjection \(\hat \theta : I \to I\).

\begin{thm} \label{SemiConjThm}
For each finite dynamical diagram \(D\), there is a geometrically fast \(X_D \subseteq \PLoI\)
such that if \(X \subseteq \HomeoI\) is geometrically fast and has dynamical diagram \(D\),
then there is a marked isomorphism \(\phi:  \gen{X} \to \gen{X_D}\) and a continuous
order preserving surjection \(\hat \theta : I \to I\) such that
\(f \hat \theta = \hat \theta \phi (f)\) for all \(f \in \gen{X}\).
\end{thm}

Theorem \ref{CombToIso} has two immediate consequences.
The first follows from the readily verifiable fact that any dynamical diagram can
be realized inside of \(F\) (see, e.g., \cite[Lemma 4.2]{CFP}).

\begin{cor} \label{EmbeddInF}
Any finitely generated
subgroup of \(\HomeoI\) which admits a geometrically fast generating set
embeds into Thompson's group \(F\).
\end{cor}
\noindent
Since by \cite{brin+squier} \(F\) does not contain nontrivial free produces of groups,
subgroups of \(\HomeoI\) which admit geometrically fast generating sets are not free products.
It should also be remarked that while our motivation comes from studying
the groups \(F\) and \(\PLoI\), the conclusion of Corollary \ref{EmbeddInF} remains valid if
\(F\) is replaced by, e.g. \(\operatorname{Diff}^\infty_+ (I)\).

\begin{cor} \label{AlgFast}
If \(\{f_i \mid i < n\}\) is geometrically fast, then
\(\gen{f_i \mid i < n}\) is marked isomorphic to
\(\gen{f_i^{k_i} \mid i < n}\) for any choice of \(k_i \geq 1\).
\end{cor}

It is natural to ask how restrictive having a geometrically fast or geometrically
proper generating set is.
The next theorem shows that many finitely generated subgroups of \(\PLoI\)
in fact do have at least a geometrically proper generating set.

\begin{thm} \label{GeoProperGen}
Every \(n\)-generated one orbital subgroup of \(\PLoI\)
either contains an isomorphic copy of \(F\) or else admits an
\(n\)-element geometrically proper generating set.
\end{thm}

\noindent
Notice that every subgroup of \(\HomeoI\) is contained in a direct product of
one-orbital subgroups of \(\HomeoI\). 
Thus if one's interest lies in studying the structure of subgroups of
\(\PLoI\) which do not contain copies of \(F\), then
it is typically possible to restrict one's attention to groups
admitting geometrically proper generating sets.
The hypothesis of not containing an isomorphic copy of \(F\) in Theorem \ref{GeoProperGen}
can not be eliminated.
This is a consequence of the following theorem and the fact that there are finite
index subgroups of \(F\) which are not isomorphic to \(F\) (see \cite{bleak+wassink}).

\begin{thm} \label{NoGeoProperGen}
If a finite index subgroup of \(F\) is isomorphic to \(\gen{X}\) for some  geometrically proper
\(X \subseteq \HomeoI\), then it is isomorphic to \(F\).
\end{thm}
\noindent
We conjecture, however, that every finitely generated subgroup of \(F\) is bi-embeddable
with a subgroup admitting a geometrically fast generating set.

While the results of this paper do of course readily adapt to \(\HomeoR \cong \HomeoI\),
it is important to keep in mind that \(\pm \infty\) must be allowed as possible
transition points when applying the definition of geometric properness and hence
geometric fastness.
For example, it is easy to establish that \(\gen{t + \sin (t), t + \cos ( t)}\) contains
a free group using the \emph{ping-pong lemma} stated above
(the squares of the generators generate a free group).
Moreover, once we define \emph{geometrically fast} in Section \ref{FastBumpsSec},
it will be apparent that the squares of the generators satisfies all of the requirements
of being geometrically fast except that it is not geometrically proper
(since, e.g., \(\infty\) is a right transition point of both functions).
As noted above, this group does not embed into \(F\) and thus
does not admit a geometrically fast (or even a geometrically proper) generating set.
See Example \ref{InfiniteBumpEx} below for a more detailed discussion of a related example.

The paper is organized as follows.
We first review some standard definitions, terminology and notation in Section \ref{DefSec}.
In Section \ref{FastBumpsSec}, we will give a formal definition of \emph{geometrically fast}
and a precise definition of what is meant by a \emph{dynamical diagram}.
Section \ref{GeoFastCriteriaSec} gives a reformulation of \emph{geometrically fast}
for finite subsets of \(\HomeoI\) which facilitates algorithmic checking.
The proof of Theorem \ref{CombToIso} is then divided between Sections
\ref{PingPongSec} and \ref{CombToIsoSec}.
The bulk of the work is in
Section \ref{PingPongSec}, which uses an analog of the \emph{ping-pong argument}
to study the dynamics of geometrically fast sets of one orbital homeomorphisms.
Section \ref{CombToIsoSec} shows how this analysis implies 
Theorem \ref{CombToIso} and how to derive its corollaries.
In Section \ref{SemiConjSec}, we will prove Theorem \ref{SemiConjThm}.
The group \(F_n\), which is the \(n\)-ary analog of Thompson's group \(F\), is
shown to have a geometrically fast generating set in Section \ref{FnSec}.
Section \ref{ExcisionSec} examines when bumps in
geometrically fast generating sets are extraneous and
can be excised without affecting the marked isomorphism type.
Proofs of Theorems \ref{GeoProperGen} and \ref{NoGeoProperGen} are given in Section \ref{GeoProperSec}.
Finally, the concept of \emph{geometrically fast}
is abstracted in Section \ref{AbstractPingPongSec},
where a generalization of Theorem \ref{CombToIso} is stated and proved, as well as corresponding embedding
theorems for Thompson's groups \(F\), \(T\), and \(V\).
This generalization in particular covers infinite geometrically fast subsets of \(\HomeoI\).
Even in the context of geometrically fast sets \(X \subseteq \HomeoI\) with only finitely many
transition points, this abstraction gives a new way of understanding \(\gen{X}\)
in terms of symbolic manipulation.

\section[Preliminaries]{Preliminary definitions, notation and conventions}

\label{DefSec}

In this section we collect a number of definitions and conventions
which will be used extensively in later sections.
Throughout this paper, the letters \(i,j,k,m,n\) will be assumed to range
over the nonnegative integers unless otherwise stated.
For instance, we will write \((a_i \mid i < k)\) to denote a sequence
with first entry \(a_0\) and last entry \(a_{k-1}\).
In particular, all counting and indexing starts at 0 unless stated otherwise.
If \(f\) is a function and \(X\) is a subset of the domain of \(f\), we will write \(f \restriction X\) to
denote the restriction of \(f\) to \(X\).

As we have already mentioned,
\(\HomeoI\) will be used to denote the
set of all orientation preserving homeomorphisms of \(I\);
\(\PLoI\) will be used to denote the set of all piecewise linear elements
of \(\HomeoI\). 
These groups will act on the right.
In particular, \(tg\) will denote the result of applying a homeomorphism
\(g\) to a point \(t\).
If \(f\) and \(g\) are elements of a group, we will
write \(f^g\) to denote \(g^{-1}fg\).

Recall that from the introduction that
if \(f\) is in \(\HomeoI\), then its \emph{support} is defined to
be \(\supt(f):=\{ t \in I \mid tf \ne t\}\).
The support of a subset of \(\HomeoI\) is the union of the supports of its
elements. 
A left (right) transition point of \(f\) is a \(t \in I \setminus \supt(f)\) such that for every \(\epsilon > 0\),
\((t,t+\epsilon) \cap \supt(f) \ne \emptyset\)
(respectively \((t-\epsilon,t) \cap \supt(f) \ne \emptyset\)).  
An \emph{orbital} of \(f\) is a component its support.
An orbital of \(f\) is \emph{positive} if \(f\) moves elements of the orbital to
the right;
otherwise it is negative.
If \(f\) has only finitely many orbitals, then the left (right) transition points
of \(f\) are precisely the left (right) end points of its orbitals.
An \emph{orbital} of a subset of \(\HomeoI\)
is a component of its support.

An element of \(\HomeoI\) with one orbital will be referred to as a \emph{bump function}
(or simply a \emph{bump}).
If a bump \(a\) satisfies that \(t a > t\) on its support,
then we say that \(a\) is \emph{positive};
otherwise we say that \(a\) is \emph{negative}.
If \(f \in \HomeoI\), then \(b \in \HomeoI\) is a \emph{signed bump of \(f\)} if
\(b\) is a bump which agrees with \(f\) on its support.
If \(X\) is a subset of \(\HomeoI\), then a bump \(a\)
is \emph{used in} \(X\)
if \(a\) is positive and
there is an \(f\) in \(X\) such that \(f\)
coincides with either \(a\) or \(a^{-1}\) on the support of \(a\).
A bump \(a\) is used in \(f\) if it is used in \(\{f\}\).
We adhere to the convention that only positive
bumps are used by functions to avoid ambiguities in some statements.
Observe that if \(X \subseteq \HomeoI\) is such that
the set \(A\) of bumps used in \(X\) is finite, then
\(\gen{X}\) is a subgroup of \(\gen{A}\).

If \((g_i \mid i < n)\) and \((h_i \mid i < n)\) are two generating sequences for groups,
then we will say that \(\gen{g_i \mid i < n}\) is \emph{marked isomorphic} to
\(\gen{h_i \mid i < n}\) if the map \(g_i \mapsto h_i\) extends to an isomorphism
of the respective groups.
If \(X\) is a finite geometrically proper subset of \(\HomeoI\), then
we will often identify \(X\) with its enumeration in which the minimum transition
points of its elements occur in increasing order.
When we write \(\gen{X}\) is marked isomorphic to \(\gen{Y}\), we are
making implicit reference to these canonical enumerations of \(X\) and \(Y\).

At a number of points in the paper it will be important to distinguish between
formal syntax (for instance words) and objects (such as group elements) to which they refer.
If \(A\) is a set, then a \emph{string}
of elements of \(A\) is a finite sequence of elements of \(A\).
The length of a string \(\str{w}\) will be denoted \(|\str{w}|\).
We will use \(\varepsilon\) to denote the string of length 0.
If \(\str{u}\) and \(\str{v}\) are two strings,
we will use \(\str{uv}\) to denote
their concatenation;
we will say that \(\str{u}\) is a \emph{prefix} of \(\str{uv}\)
and \(\str{v}\) is a \emph{suffix} of \(\str{uv}\).
If \(A\) is a subset of a group, then a \emph{word} (in \(A\))
is a string of elements of
\(A^\pm := A \cup A^{-1}\).
A \emph{subword} of a word \(\str {w}\) must preserve the order from \(\str {w}\),
but does not have to consist of consecutive symbols from \(\str {w}\).
We write \(\str {w}^{-1}\) for the formal inverse of \(\str {w}\): the product of
the inverses of the symbols in \(\str {w}\) in reverse order.

Often strings have an associated evaluation (e.g. a word represents an element of a group).
While the context will often dictate whether we are working with a string or
its evaluation, we will generally use the typewriter font (e.g. \(\str{w}\)) for strings
and symbols in the associated alphabets and standard math font (e.g. \(w\)) for the
associated evaluations.

In Section \ref{GeoProperSec}, we will use the notion of the \emph{left (right) germ}
of a function \(f \in \HomeoI\) at an \(s \in I\) which is fixed by \(f\)
(left germs are undefined at 0 and right germs are undefined at 1).
If \(0 \leq s < 1\), then define the \emph{right germ of \(f\) at \(s\)} to be the set of all \(g \in \HomeoI\) such
that for some \(\epsilon > 0\), \(f \restriction (s,s+\epsilon) = g \restriction (s,s+\epsilon)\);
this will be denoted by \(\gamma_s^+(f)\).
Similarly if \(0 < s \leq 1\), then one defines the \emph{left germ of \(f\) at \(s\)};
this will be denoted by \(\gamma_s^-(f)\).
The collections 
\[
\{\gamma_s^+(f) \mid f \in \HomeoI \textrm{ and } sf = s\}
\]
\[
\{\gamma_s^-(f) \mid f \in \HomeoI \textrm{ and } sf = s\}
\]
form groups and
the functions \(\gamma_s^+\) and \(\gamma_s^-\) are homomorphisms defined
on the subgroup of \(\HomeoI\) consisting of those functions which fix \(s\).

\section[Fast collections of bumps]{Fast collection of bumps and their dynamical diagrams}

\label{FastBumpsSec}

We are now ready to turn to the definition of \emph{geometrically fast} in the context of finite subsets of
\(\HomeoI\).
First we will need to develop some terminology.
A \emph{marking} of a geometrically proper collection of bumps \(A\)
is an assignment of a \emph{marker} \(t \in \supt(a)\) to each \(a\) in \(A\).
If \(a\) is a positive bump with orbital \((x,y)\) and marker \(t\), then we define
its \emph{source} to be the interval \(\src(a) := (x,t)\)
and its \emph{destination} to be the interval \(\dest(a) := [t a,y)\).
We also set \(\src(a^{-1}) := \dest(a)\) and \(\dest (a^{-1}) := \src(a)\).
The source and destination of a bump are collectively called its \emph{feet}. 
Note that there is a deliberate asymmetry in this definition:
the source of a positive bump is an open interval whereas the destination is half open.
This choice is necessary so that for any \(t \in \supt (a)\), there is a unique \(k\) such that
\(t a^k\) is not in the feet of \(a\), something which is a key feature of the definition.

A collection \(A\) of bumps is \emph{geometrically fast}
if it there is a marking of \(A\) for which its feet form a pairwise disjoint family
(in particular we require that \(A\) is geometrically proper).
This is illustrated in Figure \ref{GeoFastFig}, where the feet of
\(a_0\) are \((p,q)\) and \([r,s)\) and the feet of \(a_1\) are
\((q,r)\) and \([s,t)\).
\newcommand{\hashlabel}[1]
{
\xy
(0,0); (0,-2)**@{-}; (0,-4)*{\scriptstyle #1};
\endxy
}
\begin{figure}
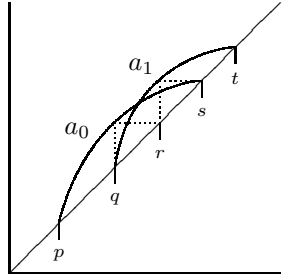

\[
\xy
(0,36); (0,0)**@{-}; (36,0)**@{-}; (0,0); (36,36)**@{-};
(16,15.6)*{\dvarbump{4}{16}{19}}; (9,19)*{a_0};
(22,21.6)*{\dvarbump{2}{14}{16}}; (17.5,27.5)*{a_1};
(14,14);(14,20)**@{.};(20,20)**@{.};
(20,25.5)**@{.};(25.5,25.5)**@{.};
(6.57,4.1)*{\hashlabel{p}};
(14.02,11.55)*{\hashlabel{q}};
(20.0,17.55)*{\hashlabel{r}};
(25.55,23.15)*{\hashlabel{s}};
(30,27.55)*{\hashlabel{t}};
\endxy
\]
\caption{A geometrically fast set of bumps}\label{GeoFastFig}
\end{figure}
Being geometrically fast is precisely the set of dynamical requirements made on the set
\(\{a_i \mid i < 3\}\) of homeomorphisms mentioned in the introduction.
We do not require here that \(A\) is finite and we will explicitly state finiteness as
a hypothesis when it is needed.
Notice however that, since pairwise disjoint families of intervals
in \(I\) are at most countable, any geometrically fast set of bumps is
at most countable.
The following are readily verified and can be used axiomatically to
derive most of the lemmas in Section \ref{PingPongSec}
(specifically Lemmas \ref{LocRedBasics}--\ref{FellowTraveler3}):
\begin{itemize}

\item for all \(a \in A^\pm\), \(\dest(a) \subseteq \supt(a)\) and if
\(x \in \supt(a)\) then there exists a \(k\) such that \(xa^k \in \dest(a)\);

\item if \(a \ne b \in A^\pm\), then \(\dest(a) \cap \dest(b) = \emptyset\);

\item if \(a \in A^\pm\) and \(x \in \supt(a)\), then
\(x a \in \dest(a)\) if and only if \(x \not \in \src(a) := \dest(a^{-1})\).

\item if \(a,b \in A^\pm\), then \(\dest(a) \subseteq \supt(b)\) or
\(\dest(a) \cap \supt(b) = \emptyset\).

\end{itemize}
This axiomatic viewpoint will be discussed further in Section \ref{AbstractPingPongSec}.

A set \(X \subseteq \HomeoI\) is geometrically fast
if it is geometrically proper and the set of bumps used in \(X\) is geometrically fast.
Note that while geometric properness is a consequence of the disjointness of the feet
if \(X\) uses only finitely many bumps, it is an additional requirement in general.
This is illustrated in the next example.

\begin{example} \label{InfiniteBumpEx}
Consider the following homeomorphism of \(\R\):
\[
t \gamma = 
\begin{cases}
3t & \textrm{ if } 0 \leq t \leq 1/2 \\
(t+4)/3 & \textrm{ if } 1/2 \leq t \leq 2 \\
t & \textrm{ otherwise}
\end{cases}
\]
Define \(\alpha , \beta \in \HomeoR\) by 
\((t+2p) \alpha = t \gamma + 2p\) and \((t+2p+1) \beta = t \gamma + 2p+1\)
where \(p \in \Z\) and \(t \in [0,2]\).
Thus the bumps used in \(\alpha\) are obtained by translating \(\gamma\) by even integers;
the bumps in \(\beta\) are the translates of \(\gamma\) by odd integers.
If we assign the marker \(1/2\) to \(\gamma\) and mark the translation of \(\gamma\) by \(p\)
with \(p + 1/2\), then it can be seen that the feet of \(\alpha\) and \(\beta\) are the intervals
\(\{(p,p+1/2) \mid p \in \Z\} \cup \{[p+1/2,p+1) \mid p \in \Z\}\),
which is a pairwise disjoint family.
Thus the bumps used in \(\{\alpha,\beta\}\) are geometrically fast.
Since \(\infty\) is a right transition point of both \(\alpha\) and \(\beta\), 
\(\{\alpha,\beta\}\) is not geometrically proper and hence not fast.
In fact, it follows readily from the formulation of the classical \emph{ping-pong lemma}
in the introduction that \(\gen{\alpha,\beta}\) is free.
\end{example}

Observe that if \(X\) is geometrically proper, each of its elements uses only finitely many bumps,
and the set of transition points of \(X\) is discrete, 
then there is a map \(f \mapsto k(f)\) of \(X\) into the positive integers such that
\(\{f^{k(f)} \mid  f \in X\}\) is geometrically fast.
To see this, start with a marking such that the closures of the sources of the bumps used in
\(X\) are disjoint;
pick \(f \mapsto k(f)\) sufficiently large so that all of the feet become disjoint.
Also notice that if \(\{f^{k(f)} \mid f \in X\}\) is geometrically fast and if \(k(f) \leq l(f)\) for \(f \in X\),
then \(\{f^{l(f)} \mid f \in X \}\) is geometrically fast as well.

If \(X\) is a geometrically fast generating set with only finitely many transition points,
then the \emph{dynamical diagram} \(D_X\) of 
\(X\) is the edge labeled vertex ordered directed graph defined as follows:
\begin{itemize}

\item the vertices of \(D_X\) are the feet of \(X\) with the order
induced from the order of the unit interval;

\item the edges of \(D_X\) are the signed bumps of \(X\) directed
so that the source (destination) of the edge is the source (destination) of the
bump;

\item the edges are labeled by the elements of \(X\) that they come from.

\end{itemize}
\noindent
We will adopt the convention that dynamical diagrams are necessarily finite.
The dynamical diagram of a generating set for the Brin-Navas group \(B\)
\cite{MR2160570} \cite{MR2135961}
is illustrated in the left half of Figure \ref{BNFigure};
the generators are \(f = a_0^{-1} a_2\) and \(g = a_1^{-1}\),
where the \((a_i \mid  i < 3)\) is the geometrically fast generating sequence illustrated in
Figure \ref{TrackingPoint}.
We have found that when drawing dynamical diagram \(D_X\) of a given \(X\), it is more
{\ae}sthetic whilst being unambiguous to collapse
pairs of vertices \(u\) and \(v\) of \(D_X\) such that:
\begin{itemize}

\item  \(v\) is the immediate successor of \(u\) in the order on \(D_X\),

\item \(u\)'s neighbor is below \(u\), and \(v\)'s neighbor is above \(v\).

\end{itemize}
Additionally, arcs can be drawn as over or under arcs to indicate their direction,
eliminating the need for arrows.
This is illustrated in the right half of Figure \ref{BNFigure}.
The result qualitatively resembles the graphs of the homeomorphisms rotated so that
the line \(y =x\) is horizontal.

An isomorphism between dynamical diagrams is a directed graph isomorphism
which preserves the order of the vertices and
induces a bijection between the edge labels
(i.e. two directed edges have equal labels before applying the isomorphism if and only if
they have equal labels after applying the isomorphism).
Notice that such an isomorphism is unique if it exists --- there is at most one order preserving bijection
between two finitely linear orders.

\begin{figure}
\[
\xy
(0,5)*{
\xymatrix{
{\bullet} &
{\bullet} &
{\bullet} \ar@/^1.5pc/[ll]^f  &
{\bullet} \ar@/^1.5pc/[rr]^f &
{\bullet} \ar@/^1.5pc/[lll]^g&
{\bullet}
}};
(-2,10); (63,10)**@{.};
(40,27)*{}; 
(50,18)*{};
\endxy
\quad
\quad
\xy
(0,5)*{
\xymatrix{
{\bullet} &
{\bullet} &
{\bullet} \ar@{-}@/^1.5pc/[ll]^f \ar@{-}@/^1.5pc/[rr]^f&
{\bullet} \ar@{-}@/^1.5pc/[ll]^g&
{\bullet} 
}};
(-2,10); (51,10)**@{.};
(40,27)*{}; 
(50,18)*{};
\endxy
\]
\caption{The dynamical diagram for the Brin-Navas generators, with an illustration
of the contraction convention \label{BNFigure}}
\end{figure}
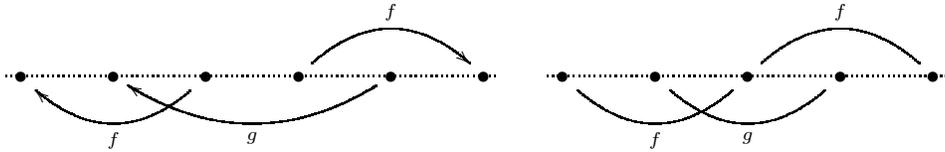
\begin{figure}
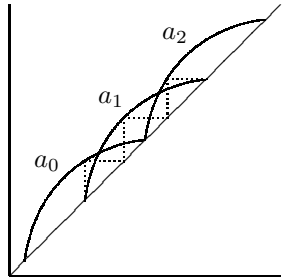

\[
\xy
(0,36); (0,0)**@{-}; (36,0)**@{-}; (0,0); (36,36)**@{-};
(10,9.6)*{\dvarbump{2}{14}{16}}; (5,15)*{a_0};
(18,17.6)*{\dvarbump{2}{14}{16}}; (13.5,23.5)*{a_1};
(26,25.6)*{\dvarbump{2}{14}{16}}; (22,32)*{a_2};
(10,10);(10,15.25)**@{.};(15.25,15.25)**@{.};(15.25,21)**@{.};(21,21)**@{.};(21,26.2)**@{.};(26.2,26.2)**@{.};
\endxy
\]
\caption{A point is tracked through a fast transition chain}\label{TrackingPoint}
\end{figure}
Observe that the (uncontracted) dynamical diagram of any geometrically
fast \(X \subseteq \HomeoI\) which has finitely many transition points has the property that
all of its vertices have total degree 1.
Moreover, any finite edge labeled vertex ordered directed graph in which each vertex has total
degree 1 is isomorphic to the dynamical diagram of some geometrically fast \(X \subseteq \HomeoI\)
which has finitely many transition points (see the proof of Theorem \ref{SemiConjThm} in Section
\ref{SemiConjSec}).
Thus we will write \emph{dynamical diagram} to mean a finite edge labeled vertex ordered
directed graph in which each vertex has total degree 1.
The edges in a dynamical diagram will be referred to as \emph{bumps} and the vertices
in a dynamical diagram will be referred to as \emph{feet}.
Terms such as \emph{source}, \emph{destination}, \emph{left/right foot} will be given the obvious meaning
in this context.

Now let \(A\) be a geometrically fast set of positive bump functions.
An element of \(A\) is \emph{isolated} (in \(A\)) if
its support contains no transition points of \(A\).
In the dynamical diagram of \(A\), this corresponds to a bump whose source and destination are
consecutive feet.
The next proposition shows that we may always eliminate isolated bumps in \(A\) by
adding new bumps to \(A\).
This will be used in Section \ref{SemiConjSec}

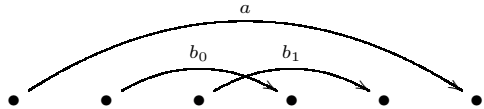
\begin{figure}
\[
\xy
\xymatrix{
{\bullet} \ar@/^2.5pc/[rrrrr]^a &
{\bullet} \ar@/^1.0pc/[rr]^{b_0} &
{\bullet} \ar@/^1.0pc/[rr]^{b_1} &
{\bullet} &
{\bullet} &
{\bullet} &
}
\endxy
\]
\caption{The bump \(a\) is made nonisolated by the addition of bumps \(b_0\) and \(b_1\).
\label{IsolatedFixFig}}
\end{figure}

\begin{prop} \label{IsolatedFixProp}
If \(A \subseteq \HomeoI\) is a geometrically fast set of positive bump functions,
then there is a geometrically fast \(B \subseteq \HomeoI\) such that \(A \subseteq B\)
and \(B\) has no isolated bumps.
Moreover, if \(A\) is finite, then \(B\) can be taken to be finite as well.
\end{prop}

\begin{proof}
If \(a \in A\) is isolated, let \(b_0\) and \(b_1\) be a geometrically fast pair of bumps with supports contained in
\(\supt(a) \setminus (\src(a) \cup \dest(a))\) such that neither \(b_0\) nor \(b_1\) is isolated in \(\{b_0,b_1\}\);
see Figure \ref{IsolatedFixFig}.
Since the feet of \(A\) are disjoint, so are the feet of \(A \cup \{b_0,b_1\}\) and \(a\) is no longer isolated
in \(A \cup \{b_0,b_1\}\).
Let \(B\) be the result of adding such a pair of bumps for each isolated bump in \(A\).
\end{proof}

\section[Criteria for geometric fastness]{An algorithmically check-able criteria for geometric fastness}
\label{GeoFastCriteriaSec}

In this section we will consider geometrically proper sets which have finitely many transition points and
develop a characterization of when they are geometrically fast.
This characterization moreover allows one to determine algorithmically when such sets are geometrically fast.
It will also provide a canonical marking of geometrically fast sets with finitely many transition points.
We need the following refinement of the notion of a \emph{transition chain}
introduced in \cite{MR2466019}.
Let \(A \subseteq \HomeoI\)
be a finite geometrically proper set of positive bump functions. 

A sequence \((a_i \mid i\le k)\) of nonisolated elements of \(A\)  is a
\emph{stretched transition chain of \(A\)} if:
\begin{enumerate}

\item \label{TransitionChain}
for all \(i<k\), \(x_i<x_{i+1}<y_i<y_{i+1}\), where 
\((x_i,y_i)\) is the support of \(a_i\);

\item \label{StretchedChain}
no transition point of \(A\) is in any interval \((x_{i+1},y_i)\).

\end{enumerate}
Notice that whether \(C\) is a stretched transition chain depends not only on the
elements of \(A\) listed in \(C\) but on the entire collection \(A\).
Since the left transition points in a stretched transition chain are strictly
increasing, we will often identify such sequences with their range.
In particular, a stretched transition chain \(C\) is \emph{maximal} if is maximal with
respect to containment when
regarded as a subset of \(A\).
We will use \(\prod C\) to denote the composition \(a_0 \cdots a_k\).
An element \(a\) of \(A\) is \emph{initial} (in \(A\)) if either \(a\) is isolated
or else the least transition point of \(A\) in the support of \(a\) is not
the right transition point of some element of \(A\).
These are precisely the elements of \(A\) which are the initial element of any
stretched transition chain which contains them.

Figure \ref{ChainDecompFig} shows a dynamical diagram, along with a list of the
maximal stretched transition chains.
\begin{figure}
\[
\xy
\xymatrix{
{\bullet} \ar@/^1.5pc/[rrr]^{a_0} &
{\bullet} \ar@/^2.5pc/[rrrrr]^{a_1} &
{\bullet} \ar@/^3.0pc/[rrrrrrr]^{a_2} &
{\bullet} \ar@/^2.0pc/[rrrr]^{a_3} &
{\bullet} \ar@/^0.5pc/[r]^{a_4} &
{\bullet} \ar@/^1.5pc/[rrr]^{a_5} &
{\bullet} &
{\bullet} &
{\bullet} &
{\bullet} &
}
\endxy
\]
\caption{The stretched transition chains in this dynamical diagram are
\(\{a_0,a_2\}\), \(\{a_1,a_5\}\), and \(\{a_3\}\); \(a_4\) is isolated
\label{ChainDecompFig}}
\end{figure}
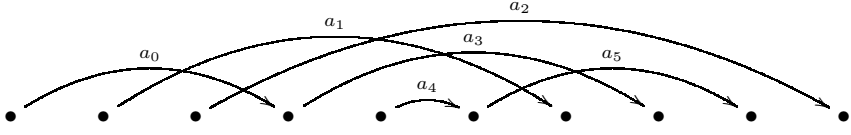

\begin{prop}\label{StretchPartProp}
If \(A\) is a finite geometrically proper set of positive bump functions in \(\HomeoI\),
then the maximal stretched transition chains in \(A\) partition the nonisolated elements of \(A\).
\end{prop}

\begin{remark}
We will see in Section \ref{FnSec} that geometrically fast stretched transition chains of length \(n\) generate
\(F_n\), the \(n\)-ary analog of Thompson's group \(F\).
Thus if \(A\) is geometrically fast, then \(\gen{A}\) is a sort of amalgam of copies of the
\(F_n\) and copies of \(\Z\).
\end{remark}

\begin{proof}
First observe that any sequence consisting of a single nonisolated element of \(A\)
is a stretched transition chain and thus is a subsequence of
some maximal stretched transition chain in \(A\).
Next  suppose that \(a\) and \(b\) are consecutive members of a stretched transition chain.
It follows that \(b\)'s left transition point is in the support of \(a\) and is
the greatest transition point of \(A\) in the support of \(a\).
In particular, \(b\) must immediately follow \(a\) in any maximal stretched transition chain.
Similarly, \(a\) must immediately precede \(b\) in any maximal stretched transition chain.
This shows that every nonisolated element of \(A\)
occurs in a unique maximal stretched transition chain of \(A\).
\end{proof}

If \(C\) is a stretched transition chain, we define \(C_{\min}\) as the
least transition point of \(A\) in the support \(C\) and
\(C_{\max}\) as the greatest transition point of \(A\) in the support \(C\).  
Since \(C\) has no isolated bumps, both \(C_{\min}\) and \(C_{\max}\) are well defined.

\begin{prop}\label{FastCriterion}
If \(A \subseteq \HomeoI\) is a finite geometrically
proper set of positive bump functions, then the following are equivalent:
\begin{enumerate}

\item \label{FastHyp}
\(A\) is geometrically fast;

\item \label{FastChainsHyp}
every stretched transition chain \(C\) of \(A\) satisfies
\(C_{\max} \leq C_{\min} \prod C\).

\item \label{MaxFastChainsHyp}
every maximal stretched transition chain \(C\) of \(A\) satisfies
\(C_{\max} \leq C_{\min} \prod C\).
\end{enumerate}
\end{prop}

\begin{remark}
This criterion for being geometrically fast was our original motivation for
the choice of the terminology: 
the dynamics of the homeomorphisms are such that
transition points can be moved to the right
through transition chains as efficiently as possible. 
This is illustrated in Figure \ref{TrackingPoint}.
\end{remark}

\begin{remark}
The choice not to allow isolated bumps to be singleton stretched transition chains is somewhat arbitrary,
although it would be necessary to make awkward adjustments to the definitions above if
we took the alternate approach.
It seems appropriate to omit them since they play no role in determining whether a group
is fast other than contributing to the collective set of transition points of the set of bumps under consideration.
\end{remark}

\begin{proof}
To see that (\ref{FastHyp}) implies (\ref{FastChainsHyp}), let \(A\) be given
equipped with a fixed marking witnessing that it is geometrically fast.
Let \(C = (a_k \mid k \leq m)\) and
let \(s_k\) denote the marker of \(a_k\).
Notice that if \(t\) is a transition point of \(A\) which is in the support of \(a_k\),
then \(s_k \leq t \leq s_k a_k\); otherwise
a foot associated to \(t\) would intersect a foot of \(a_k\).
It follows that \(s_k a_k\) is in the support of \(a_{k+1}\)
since the left transition point of \(a_{k+1}\) is in the support of \(a_k\) by
(\ref{TransitionChain}) in the definition of stretched transition chain.
Therefore the left foot of \(a_{k+1}\) is to the left of the right foot of \(a_k\)
and so \(s_{k+1} \leq s_k a_k\).
We now have inductively that \(s_m \leq s_0 a_0 \cdots a_{m-1}\) and
hence
\[
C_{\max} \leq s_m a_m \leq s_0 a_0 \cdots  a_m \leq C_{\min}  a_0 \cdots a_m.
\]

In order to see that (\ref{FastChainsHyp}) implies (\ref{FastHyp}),
let \(( a_k \mid  k \leq n )\) be the enumeration of \(A\)
such that the left transition points of the \(a_k\)'s occur in increasing order.
Let \((x_k,y_k)\) denote the support of \(a_k\) for \(k \leq n\).
Begin by setting \(s_k\) to be the least transition point of \(A\) in the support of \(a_k\);
if \(a_k\) is isolated we let \(s_k\)
be the midpoint of the support of \(a_k\).
Define \(t_k\) by induction.
If \(a_k\) is initial, we define \(t_k := s_k\).
If \(s_k\) is the right transition point of \(a_j\) for some \(j < k\),
then define \(t_k := t_j a_j\).
By induction, \(t_k < y_k\).
Observe that if \(t_k = t_j a_j\), then \(t_k = t_j a_j < s_k\)
since \(s_k\) is the right transition point of \(a_j\) and in particular is fixed by \(a_j\).

\begin{claim} \label{ChainClaim}
If \(s\) is a transition point of \(A\) in the support of \(a_k\),
then \(s \leq t_k a_k\).
\end{claim}

\begin{proof}
Let \(C\) be the stretched transition chain which finishes with \(a_k\) and
which is maximal with this property.
If \(C\) is \(\{a_k\}\), then \(t_k = s_k = C_{\min}\) and the claim follows
from our hypothesis that
\(s \leq C_{\max} \leq C_{\min} \prod C \leq t_k a_k\).
If \(a_j\) is the immediate predecessor of \(a_k\) in \(C\)
and \(C\) starts with \(a_{i}\),
then \(C_{\min} = t_i = s_i\).
Moreover, if \(C'\) is obtained from \(C\) by removing \(a_k\), then
by an inductive argument we have that \(C_{\min} \prod C' = t_i \prod C' = t_k\).
By hypothesis (\ref{FastChainsHyp}), we have
that \(C_{\max} \leq C_{\min} \prod C = t_k a_k\).
Since \(s_k\) is the right transition point of \(a_j\) which is the last entry
of \(C'\), we must
have that \(t_i \prod C' = t_k < s_k\).
Furthermore, since any transition point of \(A\) in the support of \(a_k\)
is between \(s_k\) and \(C_{\max}\) we have our desired conclusion.
\end{proof}

We now claim that the assignment \(a_k \mapsto t_k\) defines a marking which witnesses that \(A\) is geometrically fast.
We need to show that if \(i < j \leq n\), then the feet of \(a_i\) are disjoint from
those of \(a_j\).
By our assumption on the enumeration, we have that \(x_i < x_j\).
Note that if \(y_i < x_j\), the support of \(a_i\) is disjoint from the support of
\(a_j\).
In particular the feet of \(a_i\) and \(a_j\) are disjoint.

If \(x_i < x_j < y_j < y_i\),
then \(x_i < t_i \leq x_j\).
In particular, the left foot \((x_i,t_i)\)
of \(a_i\) is disjoint from the support of
\(a_j\) and hence from its feet.
Claim \ref{ChainClaim}
implies that \(y_j \leq t_i a_i \) and hence that the right foot of \(a_i\), which
is \([t_i a_i,y_i)\) is disjoint from the support of \(a_j\).

Finally, suppose that \(x_i < x_j < y_i < y_j\).
Observe that the left foot of \(a_i\) is disjoint from the support
of \(a_j\) and, by a similar argument as in the previous case,
the right foot of \(a_j\) is disjoint from the support of \(a_i\).
Thus we only need to verify that the right foot of \(a_i\) is disjoint
from the left foot of \(a_j\).
By Claim \ref{ChainClaim} 
\[
x_j < t_j \leq s_j \leq t_i a_i < y_i
\]
and hence \((x_j,t_j)\) is disjoint from \([t_ia_i,y_i)\), as desired.

It now remains to show that (\ref{MaxFastChainsHyp}) implies (\ref{FastChainsHyp});
we argue the contrapositive.
Suppose that \(C\) is a stretched transition chain and that, as a sequence, \(C\) is the
concatenation of \(C'\) and \(C''\), each of which are stretched transition chains.
Since we have seen above that every stretched transition chain is an interval within
a maximal stretched transition chain,
the desired implication will follow if we can show that
\(C_{\max} > C_{\min} \prod C\) holds if either
\(C'_{\max} > C'_{\min} \prod C'\) or \(C''_{\max} > C''_{\min} \prod C''\).
If \(C'_{\max} > C'_{\min} \prod C'\), then
\[
C_{\max} \geq C'_{\max} > C'_{\min} \prod C' =
C_{\min} \prod C' \prod C'' =C_{\min} \prod C
\]
since \(C'_{\max}\) is the least transition point of \(C''\) and hence a lower bound for the support
of \(\prod C''\).
If  \(C'_{\max} \leq C'_{\min} \prod C'\) but \(C''_{\max} > C''_{\min} \prod C''\), then
\[
C_{\max} = C''_{\max} > C''_{\min} \prod C''  \geq C'_{\min} \prod C' \prod C'' = C_{\min} \prod C
\]
since \(C''_{\min}\) is the greatest transition point of \(C'\) and thus
an upper bound for the support of \(\prod C'\).
\end{proof}

Observe that the proof of Proposition \ref{FastCriterion}
gives an explicit construction
of a marking of a family \(A\) of positive bumps.
This marking has the property that if \(A\) is geometrically fast,
then it is witnessed as such by the marking.
We will refer to this marking as the \emph{canonical marking} of \(A\).

Finally, let us note that Proposition \ref{FastCriterion} gives us a means
to algorithmically check whether a set of positive bumps \(A\) is fast.
Specifically, perform the following sequence of steps:
\begin{itemize}

\item determine whether \(A\) is geometrically proper;

\item if so, partition the non isolated elements of
\(A\) into maximal stretched transition chains;

\item for each maximal stretched transition chain \(C\) of \(A\),
determine whether \(C_{\max} \leq C_{\min} \prod C\).

\end{itemize}
This is possible provided we are able to perform the following basic queries:
\begin{itemize}

\item test for equality among the transition points of elements of \(A\);

\item determine the order of the transition points of elements of \(A\);

\item determine the truth of \(C_{\max} \leq C_{\min} \prod C\) whenever \(C\) is a stretched transition chain.

\end{itemize}

\section{The ping-pong analysis of geometrically fast sets of bumps}
\label{PingPongSec}

In this section, we adapt the ping-pong argument to the setting of fast families of
bump functions.
While the culmination will be Theorem \ref{Faithful} below, the lemmas we will develop will
be used in subsequent sections.
They also readily adapt to the more abstract setting of Section \ref{AbstractPingPongSec}.

Fix, until further notice, a (possibly infinite) geometrically fast collection \(A\) of positive bumps
equipped with a marking; in particular we will write \emph{word} to mean \emph{\(A\)-word}.
Central to our analysis will be the notion of a \emph{locally reduced word}.
A word \(\str{w}\) is \emph{locally reduced} at \(t\) if it is freely reduced and whenever
\(\str{ua}\) is a prefix of \(\str{w}\) for \(a \in A^\pm\),
\(tua \ne tu\).
If \(\str{w}\) is locally reduced at every element of a set \(J \subseteq I\), then we write that
\(\str{w}\) is \emph{locally reduced on \(J\)}.

The next lemma collects a number of useful observations about locally
reduced words; we omit the obvious proofs.
Recall that if \(\str{u}\) and \(\str{v}\) are freely reduced, then the free
reduction of \(\str{uv}\) has the form \(\str{u}_0 \str{v}_0\) where \(\str{u}=\str{u}_0 \str{w}\),
\(\str{v}=\str{w}^{-1}\str{v}_0\), and \(\str{w}\) is the longest common suffix of
\(\str{u}\) and \(\str{v}^{-1}\).
In particular, if \(\str{u}\), \(\str{v}\), and \(\str{w}\) are freely reduced words and the
free reductions of \(\str{uv}\) and \(\str{uw}\) coincide, then \(\str{v} = \str{w}\).

\begin{lemma}\label{LocRedBasics} All of the following are true:
\begin{itemize}

\item For all \(x\in I\) and all words \(\str{w}\), there is a subword
\(\str{v}\) of \(\str{w}\) which is locally reduced at \(x\) so that
\(xv=xw\).

\item For all \(x\in I\) and words \(\str{u}\) and \(\str{v}\), if \(\str{u}\) is
locally reduced at \(x\), and \(\str{v}\) is locally reduced at \(xu\), then
the free reduction of \(\str{uv}\) is locally reduced at \(x\).

\item For all \(x\in I\) and words \(\str{w}\), if \(\str{w}\) is locally
reduced at \(x\) and \(\str{w}=\str{uv}\), then \(\str{u}\) is locally reduced at
\(x\) and \(\str{v}\) is locally reduced at \(xu\).

\item For all \(x\in I\) and words \(\str{w}\), if \(\str{w}\) is locally
reduced at \(x\), then \(\str{w}^{-1}\) is locally reduced at \(xw\).

\end{itemize}
\end{lemma}
\noindent
(Recall here our convention that a \emph{subword} is not required to consist
of consecutive symbols of the original word.)
For \(x\in I\), we use \(x\gen{A}\) to denote the orbit
of \(x\) under the action of \(\gen{A}\) and for
\(S\subseteq I\), we let \(S\gen{A}\) be the union of
those \(x\gen{A}\) for \(x\in S\).

A marker \(t\) of \(A\) is \emph{initial} if whenever \(s < t\) is the marker of \(a \in A\),
then \(t \ne sa\).
If \(A\) is finite and we are working with the canonical marking, then the initial markers
are precisely the markers of the initial intervals.
Let \(M_A\) be the set of initial markers of \(A\).
We will generally suppress the subscript if the meaning is clear from the context;
in particular we will write \(M\gen{A}\) for \(M_A\gen{A}\).
Notice that every marker is contained in \(M \gen{A }\).

Aside from developing lemmas for the next section,
the goal of this section is to prove that the action of
\(\gen{A}\) on \(M\gen{A}\) is faithful.
The next lemma is the manifestation of the \emph{ping-pong argument}
in the context in which we are working.
If \(\str{w} \ne \varepsilon\) is a word, the \emph{source} (\emph{destination}) of
\(\str{w}\) is the source of the first
(destination of the last) symbol in \(\str{w}\).
The source and destination of \(\varepsilon\) are \(\emptyset\).

\begin{lemma} \label{PingPong}
If \(x \in I\) and \(\str{w} \ne \varepsilon\)
is a word which is locally reduced at \(x\), then
either \(x \in \src(\str{w})\) or \(xw \in \dest(\str{w})\).
\end{lemma}

\begin{proof}
The proof is by induction on the length of \(\str{w}\).
We have already noted
that if \(\str{w} = \str{a}\) and \(t \in \supt(a)\),
then \(ta \in \dest(a)= \dest(\str{w})\) if and only if \(t \not \in \src(a) = \src(\str{w})\).
Next suppose that \(\str{w}\) has length at least 2, \(x \not \in \src(\str{w})\),
and let \(\str{v}\) be a (possibly empty) word such that \(\str{w} = \str{vab}\)
for \(a,b \in A^\pm\).
Since \(\str{w}\) is locally reduced, \(b \ne a^{-1}\)
and thus the destination of \(a\) is not the source of \(b\).
Since \(A\) is geometrically fast, the destination of \(a\) is disjoint from the source of
\(b\).
By our inductive hypothesis, \(y = x v a\) is in
\(\dest(a) \subseteq \supt(b) \setminus \src(b)\).
Thus \(x w = x v ab  = y b\) is in the destination of \(b\).
\end{proof}

If \(\str{w}\) is a word, define \(J(\str{w}) :=\supt(a)\setminus \src(a)\)
where \(\str{a}\) is the first symbol of
\(\str{w}\).
Notice that if \(\str{w}\) is locally reduced at \(x\), then ``\(x \in J(\str{w})\)'' is equivalent to
``\(x \not \in \src(\str{w})\)''.
The following lemma is easily established by induction on the length of \(\str{w}\)
using Lemma \ref{PingPong}.

\begin{lemma} \label{threaded}
If \(\str{w}\) is a word and \(x \in J(\str{w})\), then \(\str{w}\) is locally reduced
at \(x\) if and only if \(\str{w}\) is freely reduced and 
\(\dest(a) \subseteq \supt(b)\) whenever \(\str{ab}\) are consecutive symbols in \(\str{w}\).
In particular, if \(\str{w}\) is freely reduced, then \(\str{w}\) is locally reduced on \(J(\str{w})\)
provided \(\str{w}\) is locally reduced at some element of \(J(\str{w})\).
\end{lemma}

When applying Lemma \ref{PingPong}, it will be useful to be able to assume
that \(x\) is not in \(\src(\str{w})\).
Notice that if \(x\) is not in the feet of any elements of \(A\), then this is automatically true
(for instance this is true if \(x \in M\)).
The next lemma captures an important consequence of Lemma \ref{PingPong}.

\begin{lemma}\label{Collision}
Suppose \(x_0,x_1 \in I\), \(\str{u}_i\) is locally reduced at \(x_i\) and
\(x_i \not \in \src(\str{u}_i)\).
If \(|\str{u}_0| \leq |\str{u}_1|\) and \(x_0 u_0 = x_1 u_1\),
then \(\str{u}_0\) is a suffix of \(\str{u}_1\).
In particular if \(t \in I\) is not in any of the feet of \(A\)
and \(\str{u}\) and \(\str{v}\) are words
that are locally reduced at \(t\) with \(tu=tv\), then \(\str{u}=\str{v}\).
\end{lemma}

\begin{proof}
The main part of the lemma is proved by induction on \(|\str{u}_0|\).
If \(\str{u}_0 = \varepsilon\), this is trivially true.
Next suppose that \(\str{u}_i \str{a}_i\) is locally
reduced at \(x_i\) and \(x_i \not \in \src(\str{u}_i \str{a}_i)\).
If \(x_0 u_0 a_0 = x_1 u_1 a_1\),
then by Lemma \ref{PingPong},
\(\str{a}_0 = \str{a}_1\).
We are now finished by applying our
induction hypothesis to conclude that \(\str{u}_0\) is a suffix of \(\str{u}_1\).

In order to see the second conclusion, let \(t\), \(\str{u}\) and \(\str{v}\) be given such that
\(x:= tu = tv\) and assume without loss of generality that \(|\str{u}| \leq |\str{v}|\).
By the main assertion of the lemma, \(\str{v} = \str{wu}\) for some \(\str{w}\).
Since \(tw = t\), Lemma \ref{PingPong} implies \(\str{w} = \varepsilon\).
\end{proof}

If \(x \in \src(a)\) for some \(a \in A^\pm\), then \(x \in \dest(a^{-1})\).
This suggests that we have ``arrived at'' \(x\) by applying a locally reduced word to some other point.
Moreover  \(a^{-1}\) is the unique element \(b\) of \(A^\pm\) such that \(x \in \dest(b)\).
Thus we may attempt to ``trace back'' to where \(x\) ``came from.''
This provides a recursive definition of a sequence which starts at
\(\str{a}^{-1}\) and grows to the left, possibly infinitely far.
This gives rise to the notion of a \emph{history} of a point \(x \in I\), which
will play an important role in the
proof of Theorem \ref{Faithful} below and also in Section \ref{AbstractPingPongSec}.
If \(t \in I\) is not in \(\dest(a)\) for any \(a \in A^\pm\), then we say that
\(t\) has \emph{trivial history} and define \(\tilde t := \{a \in A : t \in \supt(a)\}\).
If \(x \in I\), define \(\eta(x)\) to be the set of all strings of the following form:
\begin{itemize}

\item words \(\str{u}\) such that for some \(t \in I\),
\(tu = x\), \(\str{u}\) is locally reduced at \(t\), and \(t \not \in \src(\str{u})\);

\item strings \(\str{\tilde t u}\) such that \(t\) has trivial history, \(\str{u}\) is locally
reduced at \(t\), and \(tu = x\).

\end{itemize}
\noindent
Notice that if \(\str{w}\) is a word, then \(\str{w}^{-1}\) is in \(\eta(x)\) if and only if
\(\str{w}\) is locally reduced at \(x\) and \(xw \not \in \dest(\str{w})\).

We will refer to elements of \(\eta(x)\) as \emph{histories} of \(x\).
We will say that \(x\) has \emph{finite history} if \(\eta(x)\) is finite.
The following easily established using Lemmas \ref{LocRedBasics} and \ref{Collision};
the proof is omitted.

\begin{lemma} \label{BasicHistory}
The following are true for each \(x \in I\):
\begin{itemize}

\item
\(\eta(x)\) is closed under taking suffixes;

\item
for each \(n\), 
\(\eta(x)\) contains at most one sequence of length \(n\);

\item
If  \(\str{v}\) is a word in \(\eta(x)\), then \(\eta(xv^{-1}) = \{\str{u} : \str{uv} \in \eta(x)\}\).

\end{itemize}
\end{lemma}
\noindent
It is useful to think of \(\eta(x)\) as the suffixes of a single sequence
which is either finite or grows infinitely to the left.

In what follows, we will typically use \(s\) and \(t\) to denote elements of \(I\)
with finite history and \(x\) and \(y\) for arbitrary elements of \(I\).
The following is a key property of having a trivial history.

\begin{lemma}\label{DisjointOrbits}
If \(s\ne t\) have trivial history, then
\(s\gen{A}\) and \(t\gen{A }\) are disjoint.
\end{lemma}

\begin{proof}
If the orbits intersect, then \(t=sw\) for some word \(\str{w}\).
By Lemma \ref{LocRedBasics} we can take \(\str{w}\) to be locally reduced
at \(s\).
By Lemma \ref{PingPong}, \(sw\) is in the destination of \(\str{w}\).
But \(t = sw\) has trivial history, which is impossible.
\end{proof}

Recall that the set of
freely reduced words in a given generating set has the structure of
a rooted tree with the empty word as root and where ``prefix of'' is
synonymous with ``ancestor of.''
The \emph{ping-pong argument} discovers orbits that reflect this structure.
Define a labeled directed graph on \(I\) by putting an arc with label \(a\) from \(x\) to
\(xa\) whenever \(a \in A^\pm\) and \(xa \ne x\).
The second part of Lemma \ref{Collision} asserts that if \(x\) is in the orbit of a point \(t\) with trivial
history, then there is a unique path in this graph connecting \(t\) to \(x\).
It follows that if there is a path between two elements of \(I\) with finite history, it is unique,
yielding the following lemma.

\begin{lemma}\label{TreeAction} If
\(s,t \in I\) have finite histories,
then there is at most one word \(\str{w}\) which is
locally reduced at \(s\) so that \(sw=t\).
\end{lemma}

Notice that the assumption of finite history in this lemma is necessary.
For instance if we consider the positive bumps \(a_0\) and \(a_1\) Figure \ref{TrackingPoint},
there must be an \(x \in \supt(a_0) \cap \supt(a_1)\) such that \(xa_0 a_1^{-1} = x\).
This follows from the observation that if \(s < t\) are,
respectively, the left transition point of \(a_1\)
and the right transition point of \(a_0\), then
\[s a_0 a_1^{-1} > s a_1^{-1} =  s \qquad \textrm{ and } \qquad  t a_0 a_1^{-1} = t a_1^{-1} < t\]
which implies the existence of the desired \(x\) by applying the Intermediate Value Theorem
to \(t \mapsto t a_0 a_1^{-1} - t\).

Given two points \(x,y \in I\) and a word \(\str{w}\), it will be useful to find a single
word \(\str{w}'\) which is locally reduced at \(x\) and \(y\) and which satisfies
\(xw' = xw\) and \(yw' = yw\).
The goal of the next set of lemmas is to provide a set of sufficient conditions for the existence of such a \(\str{w}'\).
It will be convenient to introduce some additional terminology at this point.
If \(x \in I\), then
we say that \(\str{w}\) is a \emph{return word} for \(x\) if
\(xw = x\) and \(\str{w} \ne \varepsilon\);
a \emph{return prefix} for \(x\) is a prefix which is a return word.
We will see that ``\(\str{w}\) does not have a return prefix for \(s\)'' is a useful hypothesis.
The next lemma provides some circumstances under which this is true.

\begin{lemma} \label{NoReturn}
If \(s \in I\) has finite history,
\(\str{u}\) is locally reduced at \(s\),
and \(\str{w}\) is a word of length less than \(\str{u}\),
then \(\str{uw}\) has no return prefix for \(s\).
\end{lemma}

\begin{proof}
Notice that it suffices to prove that \(\str{uw}\)
is not a return word for \(s\).
If it were, then there would be a locally reduced subword
\(\str{v}\) of \(\str{w}^{-1}\) such that \(su = sv\).
Since \(|\str{v}| \leq |\str{w}| < |\str{u}|\),
this would contradict Lemma \ref{TreeAction}.
\end{proof}

\begin{lemma} \label{FellowTraveler1}
Suppose that \(\str{w} \ne \varepsilon\) is a word and \(x \in J (\str{w})\).
If \(\str{w}\) has no return prefix for \(x\) and \(\str{w'}\) is locally reduced at
\(x\) with \(xw' = xw\), then:
\begin{itemize}

\item \(J(\str{w}') = J(\str{w})\);

\item \(\str{w}'\) is locally reduced on \(J(\str{w})\);

\item if \(y \in J(\str{w})\), then \(yw' = yw\).

\end{itemize}
\end{lemma}

\begin{proof}
The proof is by induction on the length of \(\str{w}\).
If \(\str{w}\) has length \(1\), then there is nothing to show.
Suppose now that \(\str{w} = \str{ub}\) for some \(b \in A^\pm\) and \(\str{u} \ne \varepsilon\).
Let \(\str{u}'\) be locally reduced at \(x\) such that \(xu' = xu\).
By our inductive assumption, \(J:= J(\str{u}') = J(\str{u}) = J(\str{w})\),
\(\str{u}'\) is locally reduced on \(J\) and if 
\(y \in J\), then \(yu' = yu\).
By our assumption, \(xu' = xu \ne x\) and so \(\str{u}' \ne \varepsilon\).
If \(\str{u}'\str{b}\) is not freely reduced, then its free reduction \(\str{w}'\)
satisfies that \(xw' = xu'b = xub = xw \ne x\).
In particular, \(\str{w}' \ne \varepsilon\) and retains the first symbol of \(\str{u}'\).
Furthermore, since \(\str{u}'\) is locally reduced on \(J\) and since
\(\str{w}'\) is a prefix of \(\str{u}'\), \(\str{w}'\) is also locally reduced on \(J\).
Also, if \(y \in J\), then \(yw = yub = yu'b = yw'\).

Suppose now that \(\str{u}'\str{b}\) is freely reduced. 
By Lemma \ref{PingPong}, \(Ju = Ju' \subseteq \dest(\str{u}')\) for all \(y \in J\).
If \(\dest(\str{u}')\) is disjoint from \(\supt(b)\), then
\(yu' = yu'b = yub = yw\) for all \(y \in J\).
Since \(\str{u}'\) is locally reduced, we are again done in this case.
In the remaining case, \(\dest(\str{u}') \subseteq \supt(b)\) in which case
\(\str{w}' = \str{u}'\str{b}\) is locally reduced at all \(y \in J\).
Since  for all \(y \in J\), \(yu = yu' \ne yu'b = yub = yw\) we have that
\(\str{w}'\) is locally reduced on \(J\).
Clearly \(J(\str{w}') = J(\str{u}') = J(\str{w})\) and we are finished.
\end{proof}
\noindent
Lemma \ref{FellowTraveler1} has two immediate consequences which will be easier to
apply directly.

\begin{lemma} \label{FellowTraveler2}
If \(\str{w}\) is a word and
there is an \(s \in J:=J (\str{w})\) with finite history such that
\(\str{w}\) is a minimal return word for \(s\),
then \(w\) is the identity on \(J\).
\end{lemma} 

\begin{proof}
Let \(\str{w} = \str{ua}\) and let \(\str{u}'\) be locally reduced at \(s\) with \(su' = su\).
Since \(sw = su'a = s\), it must be that \(\str{u}' = \str{a}^{-1}\).
By Lemma \ref{FellowTraveler1},
\(t u' = tu\) whenever \(t \in J\).
Thus \(tw = tu' a = ta^{-1} a = t\) for all \(t \in J\).
\end{proof}

\begin{lemma} \label{FellowTraveler3}
If \(s,t \in I\) have finite histories and \(\eta(s) = \eta(t)\), then
any return word for \(s\) is a return word for \(t\).
If moreover \(s\) and \(t\) have trivial history and
\(\str{w} \ne \varepsilon\) is not a return word for \(s\),
then there is an \(a \in A^\pm\) such that \(\{sw,tw\} \subseteq \dest(a)\).
\end{lemma}

\begin{proof}
If there is a minimal return prefix of \(\str{w}\) for \(s\), then by Lemma \ref{FellowTraveler2},
it is also a return prefix for \(t\).
By iteratively removing minimal return prefixes for \(s\), 
we may assume \(\str{w}\) has no return prefixes for \(s\).
Observe that since \(\eta(s) = \eta(t)\), \(\{s,t\} \subseteq J(\str{w})\).
Lemma \ref{FellowTraveler1} now yields the desired conclusion.
\end{proof}

The following theorem shows that the restriction of the action of \(\gen{A}\) to \(M\gen{A}\) is faithful.

\begin{thm}\label{Faithful}
Suppose that \(A \subseteq \HomeoI\)
is a (possibly infinite) geometrically fast set of positive bump functions, equipped with a marking.
If \(g \in \gen{A}\) is not the identity, then
there is a \(y \in M \gen{A}\) such that \(y g \ne y\).
\end{thm}

\begin{remark}
If \(A\) is finite and equipped with the canonical marking, then
the cardinality of \(M\) is the sum of the number of maximal stretched transition chains in \(A\)
and the number of isolated elements of \(A\).
\end{remark}

\begin{proof}
First observe that if there is an \(x \in I\) such that \(xg \ne x\), then by continuity of \(g\),
there is an \(x \in I\) such that \(xg \ne x\) and \(x\) is not in the orbit of a transition point or
marker (orbits are countable and neighborhoods are uncountable).
Fix such an \(x\) and a word \(\str{w}\) representing \(g\) for the duration of the proof.
The proof of the theorem breaks into two cases, depending on whether \(x\) has
finite history.

We will first handle the case in which \(x\) has trivial history;
this will readily yield the more general case in which \(x\) has finite history.
Suppose that \(x \not \in \src(a)\) for all \(a \in A\).
Let \(y < x\) be maximal such that \(y\) is either a transition point or a marker.

\begin{claim} \label{NotSplit}
\(y\) has trivial history and
\(\tilde x = \tilde y\).
\end{claim}

\begin{proof}
First suppose that \(y\) is the marker of some \(a \in A\).
Notice that by our assumption of maximality of \(y\), the right transition point
of \(a\) is greater than \(y\).
In this case, both \(x\) and \(y\) are in the support of \(a\).
Furthermore, observe that \(y\) is not in the foot of any \(b \in A\).
To see this, notice that this would only be possible if \(y\) is in the right foot of some \(b\).
However since \(x\) is not in the right foot of \(b\),
the right transition point of \(b\) would then be less than \(x\),
which would contradict the maximal choice of \(y\).
Finally, if \(b \in A \setminus \{a\}\), the maximal choice of \(y\) implies that
\(x\) is in the support of \(b\) if and only if \(y\) is.

If \(y\) is a transition point of some \(a \in A\), then \(y\) must be the right transition point
of \(a\) since otherwise our maximality assumption on \(y\) would imply that
\(x\) is in the left foot of \(a\), contrary to our assumption that \(x\) has trivial history.
In this case neither \(x\) nor \(y\) are in the support of \(a\).
If \(b \in A \setminus \{a\}\), then our maximality assumption on \(y\) implies
that \(\{x,y\}\) is either contained in or disjoint from the support of \(b\).
To see that \(y\) has trivial history, observe that the only way a transition point
can be in a foot is for it to be the left endpoint of a right foot.
If \(y\) is the left endpoint of a right foot,
then our maximal choice of \(y\) would mean that \(x\) is also in this foot, which is contrary
to our assumption.
Thus \(y\) must have trivial history.
\end{proof}

By Claim \ref{NotSplit} and Lemma \ref{FellowTraveler3}, \(yw \ne y\).
If \(y\) is a marker, we are done.
If \(y\) is a transition point of some \(a \in A\), then as noted above it is the right transition point of \(a\).
If \(s\) is the marker of \(a\), then \(sa^k \to y\) and by continuity \(sa^k w \to yw\).
Thus for large enough \(k\), \(sa^k w \ne sa^k\).
Since \(sa^k \in M\gen{A}\), we are done in this case.

Now suppose that \(x\) has finite history and let \(\str{\tilde{t}u} \in \eta(x)\) with \(t \in I\) and \(tu = x\).
By definition of \(\eta(x)\), \(t\) has trivial history.
Since \(xw \ne x\), we have that \(tuw \ne tu\) and hence \(tuwu^{-1} \ne t\).
It follows from the previous case
that there is an \(s \in M\gen{A}\) such that \(s uwu^{-1} \ne s\).
We now have that \(y :=su\) is in \(M\gen{A}\) and satisfies \(yw \ne y\) as desired.

Finally, suppose that \(\eta(x)\) is infinite.
Let \(\str{u} \in \eta(x)\) be longer than \(\str{w}\), let \(s\) be the marker
for the initial symbol of \(\str{u}\), and set \(y := xu^{-1}\) and \(t := su\).
Since \(\str{u}\) is locally reduced at \(y\) by assumption,
Lemma \ref{threaded} implies that \(\str{u}\) is locally reduced at \(s\).
By Lemma \ref{NoReturn}, \(\str{uw}\) has no return prefix for \(s\).
Let \(\str{v}\) be locally reduced at \(s\) such that \(suw = sv\).
Applying Lemma \ref{FellowTraveler1} to \(\str{uw}\), \(s\) and \(\str{v}\), we
can conclude that \(J(\str{v}) = J(\str{uw})\), that
\(\str{v}\) is locally reduced at \(y\), and \(yuw = yv\).
Notice that since \(xw \ne x\), \(yuw = yv \ne yu\) and in particular \(\str{v} \ne \str{u}\).
By Lemma \ref{TreeAction}, we have that \(t = su \ne sv = suw = tw\).
This finishes the proof of Theorem \ref{Faithful}.
\end{proof}

We finish this section with two lemmas which concern multi-orbital homeomorphisms but 
which otherwise fit the spirit of this section.
They will be needed in Section \ref{ExcisionSec}.

\begin{lemma} \label{FindPrefix}
Suppose that \(X \subseteq \HomeoI\) is geometrically fast and equipped with a fixed marking.
Let \(A\) be the set of bumps used in \(X\) and \(s \in M\).
If \(\str{w}\) is an \(X\)-word and \(\str{u}\) is an \(A\)-word which is locally reduced at \(s\)
and satisfies \(su = sw\), then for every prefix \(\str{u}'\) of \(\str{u}\), there is a prefix \(\str{w}'\) of
\(\str{w}\) such that \(su' = sw'\). 
\end{lemma}

\begin{proof}
Let \(\str{w}_i\) be the prefix of \(\str{w}\) of length \(i\) for \(i \leq |\str{w}|\) and let
\(\str{u}_i\) be the unique word which is locally reduced at \(s\) such that
\(su_i = sw_i\).
Notice that if \(\str{u}_{i+1} \ne \str{u}_i\), then \(\str{u}_{i+1}\)
is obtained by inserting or deleting a single symbol at/from the end of \(\str{u}_i\).
It follows that all prefixes of \(\str{u}\) occur among the \(\str{u}_i\)'s.
\end{proof}

If \(X \subseteq \HomeoI\), then an element \(s\) of \(I\) is defined to have finite history with respect to \(X\)
if it has finite history with respect to the set of bumps used in \(X\).
The meaning of \emph{return word} is unchanged in the context of \(X\)-words.

\begin{lemma} \label{X_to_A}
Suppose that \(X \subseteq \HomeoI\) is geometrically fast, equipped with a fixed marking, and that \(A\)
denotes the set of bumps used in \(X\).
Let \(\str{w} \ne \varepsilon\) be an \(X\)-word and \(\str{a} \in A^\pm\) be a signed bump of
the first symbol of \(\str{w}\).
If \(s \in J: = J(\str{a})\) and \(\str{w}\) has no proper return prefix for \(s\),
then there is an \(A\)-word \(\str{v}\) which begins with \(\str{a}\) and is such that
\(v\) and \(w\) coincide on \(J\).
\end{lemma}

\begin{proof}
The proof is by induction on the length of \(\str{w}\).
Observe that the lemma is trivially true if \(\str{w}\) has length at most \(1\).
Therefore suppose that \(\str{w} = \str{u g}\) with \(g \in X^\pm\) and \(\str{u} \ne \varepsilon\).
Let \(\str{v}'\) be an \(A\)-word which begins with \(\str{a}\) and is such that \(v'\) and \(u\) agree on \(J\).
Since \(\str{u}\) has no return prefix, Lemma \ref{FindPrefix} implies that \(\str{v}'\) has no return prefix.
Let \(\str{v}''\) be a subword of \(\str{v}'\) which is locally reduced at \(s\) and satisfies
\(sv'' = sv'\).
By Lemma \ref{FellowTraveler1}, \(\str{v}''\) is locally reduced on \(J\) and
\(J(\str{v}'') = J(\str{v}') = J\).
In particular, \(J u \subseteq \dest(\str{v}'')\).
If the support of \(g\) is disjoint from \(\dest(\str{v}'')\), set \(\str{v}:=\str{v}'\).
Otherwise,
let \(b \in A^\pm\) be the signed bump of \(g\) such that \(\dest(\str{v}'') \subseteq \supt(b)\)
and set \(\str{v}:=\str{v}' \str{b}\).
Observe that \(\str{v}\) satisfies the conclusion of the lemma.
\end{proof}

\section[The isomorphism theorem]{The isomorphism theorem for geometrically fast generating sets}

\label{CombToIsoSec}

At this point we have developed all of the tools needed to prove Theorem \ref{CombToIso},
whose statement we now recall.

\begin{customthm}{\ref{CombToIso}}
If two geometrically fast sets \(X , Y \subseteq \HomeoI\) have only finitely many transition points
and have isomorphic dynamical diagrams, then the induced
bijection between \(X\) and \(Y\) extends to an isomorphism of
\(\gen{X}\) and \(\gen{Y}\) (i.e. \(\gen{X}\) is \emph{marked isomorphic} to \(\gen{Y}\)).
Moreover, there is an order preserving bijection
\(\theta :  M \gen{X} \to M \gen{Y}\) such that
\(f \mapsto f^\theta\) induces the isomorphism 
\(\gen{X} \cong \gen{Y}\).
\end{customthm}

Observe that it is sufficient to prove this theorem in the special case when
\(X\) and \(Y\) are finite geometrically fast collections of positive bumps:
if \(A\) and \(B\) are the bumps used in \(X\) and \(Y\) respectively,
then the dynamical diagrams of \(A\) and \(B\) are isomorphic and
the isomorphism of \(\gen{A}\) and \(\gen{B}\) restricts to 
a marked isomorphism of \(\gen{X}\) and \(\gen{Y}\).

Fix, for the moment, a finite geometrically fast set of positive bumps \(X\).
As we have noted, it is a trivial matter, given a word \(\str{w}\) and a \(t \in M\),
to find a subword \(\str{w}'\) which is locally reduced at \(t\) and satisfies \(t w = tw'\).
Theorem \ref{CombToIso} will fall out of an analysis of a question of
independent interest:
how does one determine \(\str{w}'\) from \(\str{w}\)
and \(t\) using only the dynamical diagram of \(X\)?
Toward this end, we define \(\str{\tilde{t}w}\) to be a \emph{local word} if \(t \in M\) and \(\str{w}\) is a word.
(Notice that \(t \mapsto \tilde t\) is injective on \(M\);
the reason for working with \(\tilde t\) is in anticipation
of a more general definition in Section \ref{AbstractPingPongSec}.)
A local word \(\str{\tilde{t}w}\) is \emph{freely reduced} if \(\str{w}\) is.
It will be convenient to adopt the convention that \(\dest(\str{\tilde{t}}) = \{t\}\) if
\(t \in M\).
Define \(\Lambda = \Lambda_X\) to be the set of all freely reduced local words \(\str{\tilde{t}w}\)
such that if \(\str{ab}\) are consecutive symbols in \(\str{\tilde{t}w}\), then
the destination of \(a\) is between \(\src(b)\) and \(\dest(b)\) in the diagram's ordering.
Notice that the assertion that \(\str{\tilde{t}w}\) is in \(\Lambda\) can be formulated as an
assertion about \(\str{w}\), the element of \(X\) which has \(t\) as a marker and
the dynamical diagram of \(X\).

Every local word \(\str{\tilde{t}w}\)
can be converted into an element of \(\Lambda\) by iteratively
removing symbols by the following procedure:
if \(\str{ab}\) is the first consecutive pair in \(\str{\tilde{t}w}\) which witnesses that it
is not in \(\Lambda\), then:
\begin{itemize}

\item if \(b = a^{-1}\) then delete the pair \(\str{ab}\);

\item if \(b \ne a^{-1}\) then delete \(\str{b}\).

\end{itemize}
Observe that since the first symbol of a local word is not removed by this procedure,
the result is still a local word.
The \emph{local reduction} of a local word \(\str{\tilde{t}w}\) is the result of applying
this procedure to \(\str{\tilde{t}w}\) until it terminates at an element of \(\Lambda\).
The following lemma admits a routine proof by induction, which we omit.

\begin{lemma} \label{LambdaLem}
Suppose that \(X\) is a geometrically fast set of positive bumps.
If \(t \in M\) and \(\str{w}\) is a word,
then \(\str{w}\) is locally reduced at \(t\) if and only if \(\str{\tilde{t}w}\)
is in \(\Lambda\).
Moreover, if \(\str{w}'\) is such that \(\str{\tilde{t}w}'\)
is the local reduction of \(\str{\tilde{t}w}\),
then \(tw'\) and \(tw\) coincide.
\end{lemma} 

Next, order \(A^\pm\) so that if \(a,b \in A^\pm\), then \(a\) is less than \(b\) if every element
of \(\dest(a)\) is less than every element of \(\dest(b)\).
Order \(\Lambda\) with the \emph{reverse lexicographic order}:
if \(\str{uw}\) and \(\str{vw}\) are in \(\Lambda\) and the last symbol of
\(\str{u}\) is less than that of \(\str{v}\), then we declare
\(\str{uw}\) less than \(\str{vw}\).
Recall that the \emph{evaluation map} on \(\Lambda\) is the function
which assigns the value \(tu\) to each string \(\str{\tilde{t}u} \in \Lambda\).
(This is well defined since \(t \mapsto \str{\tilde{t}}\) is injective on \(M\).)
This order is chosen so that the following lemma is true.

\begin{lemma} \label{RevLex}
The evaluation map defined on \(\Lambda\) is order preserving.
\end{lemma}

\begin{proof}
Suppose that \(\str{uw}\) and \(\str{vw}\) are in \(\Lambda\) and the last symbol of 
\(\str{u}\) is less than the last symbol of \(\str{v}\).
Observe that by Lemma \ref{PingPong}, the evaluation of \(\str{u}\)
is an element of its destination.
Thus if the destination of \(\str{u}\) is less than the destination of \(\str{v}\),
then this is true of their evaluations as well.
Since \(t \mapsto tw\) is order preserving, we are done.
\end{proof}

Now we are ready to prove Theorem \ref{CombToIso}.
\begin{proof}[Proof of Theorem \ref{CombToIso}]
As noted above, we may assume that \(X\) and \(Y\) are geometrically fast families of
positive bump functions with isomorphic dynamical diagrams.
By Theorem \ref{Faithful}, we know that \(\gen{X} \restriction \big(M\gen{X}\big)\)
is marked isomorphic to \(\gen{X}\); similarly \(\gen{Y} \restriction \big(M\gen{Y}\big)\) is marked
isomorphic to \(\gen{Y}\).
It therefore suffices to define an order preserving bijection \(\theta: M\gen{X} \to M\gen{Y}\)
such that \(\theta(sa) = \theta(s) \tau (a)\), where \(s \in M\gen{X}\) and \(a \in X\) and
where \(\tau : X \to Y\) is the bijection induced by the isomorphism of the dynamical diagrams of \(X\) and \(Y\).

Define \(\mu : M_X \to M_Y\) by \(\mu(s) = t\) if \(s\) is the marker for \(a \in X\) and \(t\) is the marker
for \(\tau(a) \in Y\).
Let \(\lambda\) denote the translation of local \(X\)-words into local \(Y\)-words
induced by \(\mu\) and \(\tau\).
Define \(\theta : M\gen{X} \to M\gen{Y}\) so that \(\theta(tu)\) is the evaluation of \(\lambda (\str{tu})\)
for \(\str{\tilde{t}u} \in \Lambda_X\).
This is well defined by Lemmas \ref{DisjointOrbits}, \ref{TreeAction}, and \ref{LambdaLem}.
By Lemma \ref{RevLex} and the fact that
\(\lambda\) preserves the reverse lexicographic order,
\(\theta\) is order preserving.

Now suppose that \(s \in M\gen{X}\) and \(a \in X^\pm\).
Fix \(\str{\tilde{t}u} \in \Lambda_X\) such that \(s = tu\) and let
\(\str{\tilde{t}v} \in \Lambda_X\) be the local reduction of \(\str{\tilde{t}ua}\).
Observe that on one hand \(\theta(sa) = \theta (tv)\) is the evaluation of \(\lambda(\str{\tilde{t}v})\).
On the other hand \(\theta(s)\tau(a)\) is the evaluation of \(\lambda(\str{\tilde{t}ua})\).
Since \(\lambda\) is induced by an isomorphism of dynamical diagrams, it satisfies
that \(\str{w}'\) is the local reduction of \(\str{w}\) if and only if \(\lambda(\str{w}')\) is the local
reduction of \(\lambda(\str{w})\).
In particular, \(\lambda(\str{\tilde{t}v})\) is the local reduction of \(\lambda(\str{\tilde{t}ua})\).
By Lemma \ref{LambdaLem}, these local \(Y\)-words have the same evaluation and
coincide with \(\theta(sa)\) and \(\theta(s) \tau(a)\), respectively.
This completes the proof of Theorem \ref{CombToIso}.
\end{proof}

As we noted in the introduction Theorem \ref{CombToIso} has two immediate consequences.
First, geometrically fast sets \(X = \{f_i \mid i < n\} \subseteq \HomeoI\) with finitely
many transition points are \emph{algebraically fast}:
if \(1 \leq k_i\) for each \(i < n\), then
\(\gen{f_i \mid i < n}\) is marked isomorphic to \(\gen{f^{k_i}_i \mid i < n}\).
The reason for this is that the dynamical diagrams associated to
\(\{f_i \mid i < n\}\) and \(\{f_i^{k_i} \mid i < n\}\) are isomorphic.
Second, since the dynamical diagram of any geometrically fast set with finitely many transition points
can be realized by a geometrically fast subset of \(F\) (see, e.g., \cite[Lemma 4.2]{CFP}),
every group admitting a finite geometrically fast generating set can be embedded into \(F\).
The notion of \emph{history} in the previous section is revisited in Section \ref{AbstractPingPongSec}
where it is used to prove a relative  of Theorem \ref{CombToIso}.

More evidence of the restrictive nature of geometrically fast generating sets
can be found in \cite{kim+koberda+lodha} where groups generated by stretched
transition chains \(C\) as defined in Section \ref{FastBumpsSec} are
considered under the weaker assumption
 that consecutive pairs of elements in \(C\) are geometrically fast.
Groups generated by such a \(C\) with \(n\) elements are called
{\itshape \(n\)-chain groups}.
It is proven in
\cite{kim+koberda+lodha} that every \(n\)-generated subgroup of
\(\HomeoI\) is a subgroup of an \((n+2)\)-chain group.
Another result of \cite{kim+koberda+lodha} is that for each
\(n\ge3\), there are uncountably many isomorphism types of
\(n\)-chain groups. 
By contrast, Theorem \ref{CombToIso} (with Corollary \ref{Excision} below) implies that
the number of isomorphism types of groups with finite, geometrically fast generating sets is
countable because the number of isomorphism types of dynamical
diagrams is countable.

\section[semi-conjugacy]{Minimal representations of
geometrically fast groups and topological semi-conjugacy}

\label{SemiConjSec}

Theorem \ref{CombToIso} partitions the subgroups of \(\HomeoI\)
generated by geometrically fast sets with finitely many transition points: two such sets are
considered equivalent if their dynamical diagrams are isomorphic.
In this section we show that each class contains a (nonunique)
representative \(Y\) so that for each \(X\) in the class there is a
marked isomorphism \(\phi : \gen{X} \to \gen{Y}\) which is
induced by a semi-conjugacy on \(I\).
Specifically, the bijection \(\theta : M\gen{X} \to M\gen{Y}\) of Theorem \ref{CombToIso}
extends to a continuous order preserving surjection \(\hat{\theta}:I \to I\)
so that for all \(f\in \gen{X}\) we have \(f \hat{\theta} =\hat{\theta}\phi(f)\).
Notice that in this situation, the graph of \(\phi(f)\) is the image of the graph of \(f\) under the
transformation \((x,y) \mapsto (x\hat\theta,y\hat\theta)\).
We will refer to such a \(Y \subseteq \HomeoI\) as \emph{terminal}.
Theorem \ref{SemiConjThm} can now be stated as follows.

\begin{thm}
Each dynamical diagram \(D\) can be realized by a terminal \(X_D \subseteq \PLoI\).
\end{thm}

\begin{proof}
As in the proof of Theorem \ref{CombToIso}, it suffices to prove the theorem
under the assumption that all bumps in \(D\) are positive and all labels are distinct.
Furthermore, by Proposition \ref{IsolatedFixProp}, we may assume that \(D\) has no isolated bumps.
Let \(n\) denote the number of bumps of \(D\), set \(\ell := 1/(2n)\) and 
\[
\mathcal{J} : = \{ [i \ell, (i+1) \ell) \mid 0 \leq i < 2n \},
\]
observing that \(\mathcal{J}\) has the same cardinality as the set of feet of \(D\).
Order \(\mathcal{J}\) by the order on the left endpoints of its elements.
If \(i < 2n\), we will say that the \(i\Th\) interval in \(\mathcal{J}\) \emph{corresponds} to the \(i\Th\) foot
of \(D\).

For \(i < n\), let \((x_i,s_i)\) and \((t_i,y_i)\) be the intervals in \(\mathcal{J}\) which correspond to the left and right feet of
the \(i\Th\) bump of \(D\), respectively.
Note that since \(D\) has no isolated bumps, \(s_i < t_i\).
Define \(b_i\) to be the bump which has support \((x_i,y_i)\), maps \(s_i\) to \(t_i\) and is linear on
\((x_i,s_i)\) and \((s_i,y_i)\) --- see Figure \ref{PLBumpFig}.
\begin{figure}
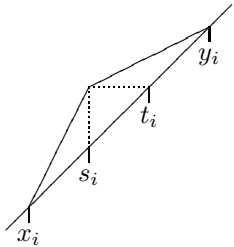

\[
\xy
(0,0); (30,30)**@{-};
(3,3); (11,19)**@{-}; (27,27)**@{-}; (11,11); (11,19)**@{.};
(19,19)**@{.};
(3,3); (3,1)**@{-}; (3,-1)*{x_i};
(11,11); (11,9)**@{-}; (11,7)*{s_i};
(19,19); (19,17)**@{-}; (19,15)*{t_i};
(27,27); (27,25)**@{-}; (27,23)*{y_i};
\endxy
\]
\caption{The function \(b_i\).\label{PLBumpFig}}
\end{figure}
If we assign \(b_i\) the marker \(s_i\), then 
the feet of \(b_i\) are either in \(\mathcal{J}\) or are the interior of an element of \(\mathcal{J}\);
in particular, the feet of \(X_D = \{b_i \mid i < n\}\) are disjoint.
Thus \(X_D\) is geometrically fast and has dynamical diagram isomorphic to \(D\).

Notice that the feet \((x_i,s_i)\) and \([t_i,y_i)\) of \(b_i\)
are each intervals of length \(\ell\) contained in
\(I\) while the middle interval \([s_i,t_i)\) is of length \(m\ell\) for some positive integer \(m\).
Moreover, since \(D\) has no isolated bumps,
there is an interval of \(\mathcal{J}\) between \(s_i\) and \(t_i\);
in particular, \(t_i - s_i \geq \ell\).
It follows that the slope of the graph of \(b_i\)
on its source \((x_i,s_i)\) and the slope of the graph of \(b_i^{-1}\)
on its source \([t_i,y_i)\) are both at least 2.

\begin{claim}\label{OrbitDensity}
If \(X_D\) is the set of positive bumps constructed above, then \(M\gen{X_D}\) is dense in \(I\).
\end{claim}

\begin{remark}
Note that we are working under the assumption that \(D\) has no isolated elements.
If \(D\) has isolated bumps, then \(M\gen{X_D}\) can not be dense.
\end{remark}

\begin{proof} 
Since every transition point of \(X_D\) is in the closure of \(M\gen{X_D}\),
it suffices to show that if \(0 \leq p < q \leq 1\), then 
\((p,q)f\) contains an endpoint of an interval in \(\mathcal{J}\) for some \(f \in \gen{X_D}\).
The proof is by induction on the minimum \(k \geq 0\) such that \(\ell 2^{-k} < q-p\).
Observe that if \(k = 0\), then \(q-p > \ell\) and thus \(q-p\) must contain an endpoint of
an interval in \(\mathcal{J}\).

Next observe that if \((p,q)\) does not contain an endpoint of an element of \(\mathcal{J}\),
then \((p,q)\) is contained in the foot of some \(b_i\) for \(i < n\).
If \((p,q) \subseteq \src(b_i)\), then since the derivative of \(b_i\) is at least 2 on it source,
it follows that  \((p,q)b_i\) is at least twice as long as \((p,q)\).
By our induction hypothesis, there is an \(f \in \gen{X_D}\) such that
\((p,q)b_i f\) contains an endpoint of \(\mathcal{J}\).
Similarly, if \((p,q) \subseteq \src(b_i^{-1})\), then \((p,q)b_i^{-1}\) is at least twice
as long as \((p,q)\) and we can find an \(f \in \gen{Y}\) such that
\((p,q)b_i^{-1} f\) contains an endpoint of \(\mathcal{J}\).
\end{proof}

In order to see that \(X_D\) is terminal, let
\(X \subseteq \HomeoI\) be geometrically fast, have finitely many transition points, and have a dynamical diagram
isomorphic to \(D\).
Let \(\theta : M\gen{X} \to M\gen{X_D}\) be order preserving and satisfy that
\(t f \theta = t \theta \phi(f)\) for all \(f \in M\gen{X}\),
where \(\phi : \gen{X} \to \gen{X_D}\) is the marked isomorphism.
Define \(\hat \theta : I \to I\) by
\[
x \hat \theta := \sup \{t\theta \mid t \in M\gen{X} \textrm{ and } t \leq x\}
\]
where we adopt the convention that \(\sup \emptyset = 0\).
Clearly \(\hat \theta : I \to I\) is order preserving and extends \(\theta\).
In particular its range contains \(M\gen{X_D}\), which by
Claim \ref{OrbitDensity} is dense in \(I\).
It follows that \(\hat \theta\) is a continuous surjection (\emph{any} order preserving map from \(I\) to \(I\) with
dense range is a continuous surjection).
That \(x f \hat \theta = x \hat \theta \phi (f)\) follows from
the fact that this is true for \(x \in M\gen{X}\) and from the continuity of \(f\) and \(\phi(f)\).
This completes the proof that \(X_D\) is terminal.
\end{proof}

\section{Fast generating sets for the groups \(F_n\)}
\label{FnSec}

In this section we will give explicit generating sets
for some well known variations of Thompson's group \(F\).
First notice that since \((0,1)\) is homeomorphic to \(\R\) by an order preserving map,
all of the analysis of geometrically fast subsets of \(\HomeoI\) transfers to \(\HomeoR\)
(with the caveat that \(\pm \infty\) must be considered as possible
transition points of elements of \(\HomeoR\); see Example \ref{x+1_x^3_ex} below).
Fix an integer \(n\ge2\).
For \(i < n\), let \(g_i\) be a homeomorphism from \(\R\) to itself defined by:
\[
tg_i := \begin{cases}
t & \textrm{if } t\le i \\
i+n(t-i) & \textrm{if } i\le t\le i+1 \\
t+(n-1) & \textrm{if } i+1 \le t.
\end{cases}
\]
In words, \(g_i\) is the identity below \(i\), has constant slope
\(n\) on the interval \([i,i+1]\), and is translation by \(n-1\)
above \(i+1\).
We will use \(F_n\) to denote \(\gen{g_i \mid i < n}\).
The group \(F_2\) is one of the standard
representations of Thompson's group \(F\).
(The more common representation of \(F\) is as a set of piecewise linear
homeomorphisms of the unit interval \cite[\S1]{CFP}.)

The groups \(F_n\) \((n \ge 2)\)  are discussed in
\cite[\S4]{brown:finiteprop} where \(F_n\) is denoted
\(F_{n,\infty}\), and in \cite[\S2]{brin+fer} where \(F_n\) is
denoted \(F_{n,0}\).
The standard infinite presentation of \(F_n\) is given in \cite[Cor. 2.1.5.1]{brin+fer}.
It follows easily from that
presentation that the commutator quotient of \(F_n\) is a free abelian group of rank \(n\).
In particular the \(F_n\)'s are pairwise nonisomorphic.

We now describe an alternate generating set for \(F_n\) which 
consists of \(n\) positive  bump functions and is geometrically fast.
For \(i<n-1\), set \(h_i : = g_i g_{i+1}^{-1}\) and let
\(h_{n-1}\) denote \(g_{n-1}\).
It is clear that \(C = \{h_i \mid i<n\}\) generates \(F_n\).
We claim that \(C\) is a geometrically fast stretched transition chain.
If \(i < n-1\), then \(h_i\) is the identity outside \([i,i+2]\).
On that interval it is a positive, one bump function since chain
rule considerations show that \(h_i\) has slope \(n\) on
\([i,i+\frac 1 n]\), slope one on \([i+\frac 1 n, i+1]\) and slope
\(1/n\) on \([i+1, i+2]\). 
Thus the support of \(h_i\) is \((i,i+2)\) and
the support of \(h_{n-1}\) is \((n-1, \infty)\).
In particular, \(C\) forms a stretched transition chain.

Next we will show that \(C\) is geometrically fast.
Observe that for \(0 < j \leq n\) that we have:
\[ 
\left( i+\frac{j}{n} \right) h_i
= 
(i+1) + \frac{j-1}{n}.
\]

For \(0\le i\le n\) set
\(t_i := i+ \frac{n-i}{n}\).
It follows from the above computation that for \(0\le i<n-1\), we
have \(t_ih_i = t_{i+1}\) so that \(t_0=1\) and \(t_n=n\).  This
makes \(t_0\) the leftmost transition point in the support of
\(h_0\) and \(t_n\) the rightmost transition point in the support of
\(h_{n-1}\).
By Proposition \ref{FastCriterion}, \(C\) is geometrically fast.
Combining this with Theorem \ref{CombToIso}, we now have that
any geometrically fast stretched transition chain generates a copy of \(F_n\).

\begin{remark}
The above proof that \(F_n\) is isomorphic to the group generated by a geometrically fast
stretched transition chain of length \(n\) is not an efficient
way to reach this conclusion.
There is a more straightforward
argument based on the standard presentation of \(F_n\) as indicated
in the proof of \cite[Proposition 1.11]{kim+koberda+lodha}.
We included this example as an introduction to the
variety of isomorphism classes in groups generated by geometrically fast sets.  
\end{remark}

\section{Excision of extraneous bumps in fast generating sets}

\label{ExcisionSec}

Sometimes fast generating sets use bumps which do not affect the marked isomorphism type
of the resulting group.
This section gives a sufficient condition for when those bumps
can be excised while preserving the marked isomorphism type.
We make no finiteness assumptions in this section.

If \(f \in \HomeoI\) and \(E\) is a set of positive bumps, then we define \(f/E \in \HomeoI\) to be the function
which agrees with \(f\) on  
\[
I \setminus \bigcup \{\supt (a) : a \in E \textrm{ and } a \textrm{ is used in } f\}
\]
and is the identity elsewhere.
If \(X\) is a fast generating set and \(E\) is a set of positive bumps, we 
define \(X/E := \{f/E : f \in X\}\).

Now let \(X\) be a fast generating set and let \(J \subseteq I\) be an interval.
We say a set \(E\) of bumps is \emph{extraneous in \(X\) as witnessed by \(J\)} if
for some \(f \in X\):
\begin{itemize}

\item every element of \(E\) is an isolated bump used in \(f\) whose
support is contained in \(J\);

\item there is a bump \(a\) used in \(f\) not in \(E\) such that \(J\) contains a foot of \(a\);

\item \(J\) is disjoint from the feet of all \(g \in X \setminus \{f\}\).

\end{itemize}
Observe that if \(E\) is extraneous in \(X\), then \(g \mapsto g/E\) preserves the canonical ordering of \(X\).

\begin{thm} \label{Excision}
If \(X \subseteq \HomeoI\) is a (possibly infinite) geometrically fast set and
\(E\) is an extraneous set of bumps used in \(X\),
then the map \(g \mapsto g/E\) extends to an isomorphism between \(\gen{X}\) and \(\gen{X/E}\).
\end{thm}

\begin{proof}
Fix \(X\) and \(E\) as in the statement of the theorem and
let \(J\) be an interval witnessing that \(E\) is extraneous
and \(f \in X\) be the element of \(X\) such that the elements of \(E\) occur in \(f\).
Let \(A\) denote the set of bumps used in \(X\),
\(M\) denote \(M_X\), and \(\widetilde M\) denote \(M_{X/E}\).
Set \(S :=M \gen{X} = M \gen{A}\) and
\(\widetilde S := \widetilde M \gen{X/E} = \widetilde M \gen{A \setminus E}\).
Observe that since the elements of \(E\) are isolated in \(A\) and since
every element of \(\widetilde M \gen{A \setminus E} \setminus \widetilde M\) is in \(\dest(b)\) for some
\(b \in (A \setminus E)^{\pm}\), it follows that \(\widetilde M \gen{A \setminus E} = \widetilde M \gen{A}\).
In particular, \(\widetilde S\) is \(\gen{A}\)-invariant and hence \(g \mapsto g \restriction \widetilde S\)
defines a homomorphism of \(\gen{X}\) to \(\gen{X} \restriction \tilde S\).

Recall that, by Theorem \ref{Faithful}, \(g \mapsto g \restriction S\) defines an isomorphism
between \(\gen{X}\) and \(\gen{X} \restriction S\).
Similarly, \(g/E \mapsto g/E \restriction \widetilde S\) defines an isomorphism between \(\gen{X/E}\)
and \(\gen{X/E} \restriction \widetilde S\).
Since \(g \restriction \widetilde S = (g/E) \restriction \widetilde S\), 
it suffices to show that \(g \mapsto g \restriction \widetilde S\) is injective on \(\gen{X}\).
That is, if \(\str{w}\) is an \(X\)-word, \(s\) is a marker for an \(a \in E\),
and \(suw \ne su\) for some \(X\)-word \(\str{u}\),
then there is a \(t \in \widetilde S\) such that \(tuw \ne tu\).
Equivalently we need to show that for every marker \(s\) of an \(a \in E\) and \(X\)-word \(\str{w}\),
if  \(sw \ne s\), then there is a \(t \in \widetilde S\) with \(tw \ne t\).
To this end let such \(s\), \(a\), and \(\str{w}\) be given with \(sw \ne s\).

\begin{claim}
If there is an \(x \in \supt(a)\) such that \(xw \not \in \supt(a)\), then there is a \(t \in \widetilde S\)
such that \(tw \ne t\).
\end{claim}

\begin{proof}
Suppose that \(xw \not \in \supt(a)\) for some \(x \in \supt(a)\).
Let \(\str{uv}\) be a locally reduced at \(x\) such that \(xuv = xw\) with \(\str{u}\) maximal
such that \(xu \in \supt(a)\);
notice that \(\str{v} \ne \varepsilon\).
Since \(a\) is isolated, \(xu \not \in \src(\str{v})\) and thus by Lemma \ref{PingPong}
\(xuv \in \dest(\str{v})\).
It follows that \(w\) must move an endpoint of \(\dest(\str{v})\), both of which are in
the closure of \(\widetilde S \cap \dest(\str{v})\).
Hence \(w\) moves an element of \(\widetilde S\). 
\end{proof}

We may therefore assume that \(sw \ne s\) but that \(w\) maps \(\supt(a)\) to \(\supt(a)\).
Let \(\str{w} = \prod_{i < n} \str{w}_i\) be a factorization of \(\str{w}\) into minimal
positive length words which define maps from \(\supt(a)\) into \(\supt(a)\).
Let \(b\) be the bump used in \(f\) with \(b \not \in E\) and such that
\(J\) contains a foot of \(b\).
Let \(t \in J \cap \widetilde S\) be sufficiently close to the transition point of \(b\) which is in \(J\)
such that for all \(k \in \Z\) with \(|k| \leq |\str{w}|\), \(t f^k \in J\).

\begin{claim}
For all \(i < n\), there is a \(p \in \{1,0,-1\}\) such that
 \(w_i \restriction J = f^p \restriction J\).
\end{claim}

\begin{proof}
If \(\str{w}_i\) begins with \(f\) or \(f^{-1}\), then it must have length \(1\) and there is nothing to
show.
If \(\str{w}_i\) begins with \(\str{g} \ne \str{f}^{\pm 1}\),
then either \(J\) is contained in or disjoint from the support of
\(g\) and it is disjoint from all of the feet of \(g\).
If \(g \restriction J\) is the identity, then again \(\str{w}_i\) has length 1.

Now suppose that \(J\) is contained in the support of \(g\) where \(g\) is neither \(f\) nor \(f^{-1}\) and
has \(s\) in its support.
Notice that since \(J\) is disjoint from the feet of \(g\),
\(J\) is contained in the support of a single bump \(c \in A^\pm\) of \(g\).

We will now show that \(sw_i = s\).
Let \(\str{gv}\) be a minimal prefix of \(\str{w}_i\) such that \(sgv \in \supt(a)\).
By Lemmas \ref{FellowTraveler1} and \ref{X_to_A}, there is an \(A\)-word \(\str{u}\) which 
is locally reduced on \(Jc = Jg\) and such that \(u\) and \(v\) coincide on \(Jg\).
Since \(\str{c}\) is locally reduced at \(s\) and \(\str{u}\) is locally reduced at \(sc\),
the free reduction of \(\str{cu}\) is locally reduced at \(s\).
If this free reduction is not \(\varepsilon\), then it must be that \(\str{cu}\) was freely reduced and
\(scu \in \dest(\str{cu})\).
Since \(scu = sgv \in \supt(a)\), \(\str{cu}\) would have to end in \(\str{a}\) or \(\str{a}^{-1}\).
But then if \(\str{u}'\) is the result of deleting the last symbol of \(\str{u}\),
then \(scu'\) is in \(\supt(a)\) and, by Lemma \ref{FindPrefix}, coincides with
\(sgv'\) for some proper prefix \(\str{v}'\) of \(\str{v}\).
This contradicts the minimal choice of \(\str{gv}\).
It must therefore have been that \(\str{u} = \str{c}^{-1}\).
It follows that \(c^{-1}\) and \(v\) coincide on \(Jg\) and thus
\(gv\) is the identity on \(J\).
By minimality of \(\str{w}_i\), \(\str{w}_i = \str{v}\) and thus \(\str{w}_i\) is the identity on \(J\). 
\end{proof}

Applying the claim, there are \(p_i \in \{1,0,-1\}\) for \(i < n\) such that:
\[
sw = s\prod_{i < n} f^{p_i}
\qquad \qquad 
tw = t \prod_{i < n} f^{p_i}.
\]
In particular, since \(sw \ne w\), it follows that \(tw \ne t\).
\end{proof}

\begin{cor}
If \(X \subseteq \HomeoI\) is geometrically fast and finite, then there is a
geometrically fast \(Y \subseteq \HomeoI\) which has finitely many transition points such that
\(\gen{X}\) is marked isomorphic to \(\gen{Y}\).
\end{cor}

\begin{proof}
Let \(X\) be given and let \(\mathcal{J}\)  consist
of the maximal intervals \(J \subseteq I\) such that for some \(f \in X\) (dependent on \(J\)):
\begin{itemize}

\item \(J\) contains at least one transition point of \(f\);

\item \(J\) contains no transition points of \(X \setminus \{f\}\);

\item \(J\) is disjoint from the feet of \(X \setminus \{f\}\).

\end{itemize}
Observe that since \(X\) is geometrically proper, \(\mathcal{J}\) is finite.
The proof is by induction on the number of elements of \(\mathcal{J}\) which contain infinitely many
transition points of \(X\).
Suppose that \(J \in \mathcal{J}\) contains infinitely many transition points of some \(f \in X\).
Let \(E\) consist of all but one of the isolated bumps of \(f\) with support contained in \(J\).
Notice that \(J\) witnesses that \(E\) is extraneous in \(X\).
By Theorem \ref{Excision}, \(\gen{X/E}\) is marked isomorphic to \(\gen{X}\).
We are now finished by our induction hypothesis.
\end{proof}

\section[Geometrically proper generating sets]{The existence of geometrically proper generating sets}

\label{GeoProperSec}

In this section we consider the question of when finitely generated subgroups of
\(F\) admit geometrically proper generating sets.
Our first task will be to prove:

\begin{customthm}{\ref{GeoProperGen}}
Every \(n\)-generated one orbital subgroup of \(\PLoI\)
either contains an isomorphic copy of \(F\) or else admits an
\(n\)-element geometrically proper generating set.
\end{customthm}
\noindent
In this section we will use the standard embedding of Thompson's group \(F\) in \(\PLoI\)
and we will make use of the homomorphism \(\pi:\PLoI \rightarrow \R \times \R\)
defined by \(\pi(f) := \big(\log_2 \left(f'(0)\right),\log_2\left(f'(1)\right)\big)\).
If \(G\) is a subgroup of \(\PLoI\), we will write \(\pi_G\) for \(\pi \restriction G\).
It is well known that \(F'\) is exactly
the kernel of \(\pi_F\) and that \(F'\) is simple.
We will need the following lemma which combines the main result of \cite{brin:ubiq} and
Lemma 3.11 of \cite{MR2383051}.

\begin{lemma} \label{CyclicGerm}
Let \(G\) be a finitely generated subgroup of \(\PLoI\) with connected support into which 
\(F\) does not embed.
The image of the homomorphism \(\pi_G\) is either trivial or cyclic.
\end{lemma}

\begin{proof}[Proof of Theorem \ref{GeoProperGen}]
Let \(G\) be a one orbital subgroup of \(\PLoI\) which
is generated by a finite set \(X\).
By conjugating by an appropriate affine homeomorphism of \(\R\),
we may assume that the support of \(G\) is \((0,1)\).
We assume that \(F\) does not embed in \(G\).
By Lemma \ref{CyclicGerm}, the image of
\(\pi_G\) must then be isomorphic to \(\Z\).
Let \(\varpi\) be the composition of \(\pi_G\) with this isomorphism.
If there is more than one element of \(X\) not in the
kernel of \(\varpi\), then we apply the Euclidean algorithm and
repeatedly reduce \(\sum \left\{ |\varpi(f)| \mid f \in X\right\}\) by replacing at each stage some
\(g\in X\) by \(gh^\epsilon\), for some suitably chosen \(h \in X\) and \(\epsilon\in\{-1,1\}\), where  \(|\varpi(h)|\le |\varpi(g)|\).
Note that this
replacement does not change the cardinality of \(X\).  Thus we can
assume that there is only one element \(f\) of \(X\) not in the
kernel of \(\varpi\).

If \(X = \{g_i \mid i < n\}\), define \(N(X)\) to be the number of
pairs \((j,x)\) such that for some \(i < j\) either \(x\) is left
transition point of both \(g_i\) and \(g_j\) or a right transition
point of both \(g_i\) and \(g_j\).
Observe that \(N(X)\) is finite since \(X\) has only
finitely many transition points and that \(X\) is geometrically
proper precisely if \(N(X) = 0\).
Notice also that if \((j,x)\) is a pair counted by \(N(X)\),
then \(x\) can be neither 0 nor 1.

Assume that \(N(X) > 0\), and let \(x\) be a point that contributes
to \(N(X)\) with \(g\ne h\) two functions in \(X\) with \(x\) a
common left transition point or common right transition point of
both \(g\) and \(h\).
Since \(x\in (0,1)\), our assumption on the support of \(G\)
implies that there is an \(f\in X\) with \(xf \ne x\).
Let \(T\) be the set of transition points of
\(g\) that are moved by \(f\).
Since each point in \(T\) is moved to infinitely
many places by powers of \(f\) and since the set of
transition points of elements of \(X\) is finite, we can find a
power \(f^k\) of \(f\) so that no element of \(T{f^k}\) is a
transition point of an element of \(X\). 
Let \(X'\) be obtained from \(X\) by replacing \(g\) by \(g^{f^k}\).
We have arranged that \(|X'|=|X|\),
that \(\gen{X'} = \gen{X}=G\),
that there is only one element of \(X'\) outside the kernel of \(\varpi\),
and that \(N(X')<N(X)\).
By induction \(\gen{X'} = \gen{X}\)
admits a geometrically proper generating set.
\end{proof}

The assumption that the support of \(G\) be connected in Theorem \ref{GeoProperGen}
turns out to be necessary as the next example shows.
Before proceeding we will develop some notation which will be helpful in proving
Theorem \ref{NoGeoProperGen} as well.
For the remainder of the section, fix two elements \(a\) and \(b\) of \(F\) which are
one bump functions whose supports are, respectively, \(\left(0,\frac{1}{2}\right)\)
and \(\left(\frac{1}{2}, 1\right)\) and which satisfy \(a'(0)=b'(1)=2\).
Notice that \(a\) is a positive bump and \(b\) is a negative bump.

\begin{example}
The group
\(\gen{ ab, a^{-1} b }\) is isomorphic to \(\Z\times\Z\) but has
no geometrically proper generating set.
We leave the details as an exercise for the reader.
\end{example}

\begin{example} \label{x+1_x^3_ex}
It was shown in \cite{x+1_x^p_free} that
subgroup \(\gen{t+1,t^3}\) of \(\HomeoR\)
is free.
Since any minimal generating set for a free group is a free basis for the group,
any minimal generating set is algebraically fast.
On the other hand this group is not embeddable
into \(F\) by \cite{brin+squier} and hence by Corollary \ref{EmbeddInF}
is not isomorphic to a group with a geometrically proper generating set.
\end{example}

Next we will prove Theorem \ref{NoGeoProperGen},
which shows that the ``\(F\)-less'' hypothesis in Theorem \ref{GeoProperGen} can not be removed.
Note that no assumption is made in this theorem concerning the number of bumps used in 
the geometrically proper generating set.

\begin{customthm}{\ref{NoGeoProperGen}}
If a finite index subgroup of \(F\) is isomorphic to \(\gen{X}\) for some  geometrically proper
\(X \subseteq \HomeoI\), then it is isomorphic to \(F\).
\end{customthm}

\begin{proof}
Let \(G\) be a subgroup of \(F\) of finite index such that
\(G \cong \gen{X}\) for some finite geometrically proper \(X \subseteq \HomeoI\).
Let \(H\) denote \(\gen{X}\) and \(\phi : G \to H\) be a fixed isomorphism.
We must show that \(G\) is isomorphic to \(F\).

By \cite{bleak+wassink} we know that, since \(G \leq F\) is of
finite index, it is
of the form \(\pi^{-1}(K)\) for some finite
index subgroup \(K\) of \(\Z\times\Z\).
Furthermore, also by \cite{bleak+wassink},
\(\pi^{-1}(K)\) is isomorphic to \(F\) if and only if \(K\) admits a generating set of the form \(\left\{\left(p,0\right),\left(0,q\right)\right\}\) for
\(p,\,q\in\Z\setminus\{0\}\).

Observe in any case that \(K\) admits a two element generating
set \(\{(p_i,q_i) \mid  i < 2\}\), as it is finite index in \(\Z\times\Z\).
We will first derive some properties of \(G\) which come from viewing it
as a subgroup of \(F\).
Define \(f_i := a^{p_i} b^{q_i}\).
It is shown in \cite{bleak+wassink} that \(G\) can be generated by
\(\{f_ i \mid i < 2\} \cup F'\).
For \(0\le x<y\le 1\) with both \(x\) and \(y\) in
\(\Z[1/2]\) we let \(F_{[x,y]}\) denote the subgroup of \(F\)
consisting of those elements of \(F\) whose support is contained in
\([x,y]\).  It is standard that \(F_{[x,y]}\cong F\) and that if
\(x<w<y<z\), then \(F_{[x,y]}\cup F_{[w,z]}\) generates
\(F_{[x,z]}\).

Next we claim that the kernel of \(\pi_G\) is \(G' = F'\).
It is trivial that \(F'\subseteq G\) and since \(F'\) is simple and not abelian
it follows that we also have \(F'\subseteq G'\).
On the other hand, since \(G\subseteq F\) we have \(G'\subseteq F'\).
Since the image of \(\pi_G\) is abelian, we have
\(G'\) is contained in the kernel of \(\pi_G\).
Lastly, \(\ker(\pi_G) \subseteq \ker(\pi) \subseteq F'\). 

\begin{claim} \label{BothGermsNontrivial}
If \(g \in G\) is such that neither coordinate of \(\pi (g)\)
is \(0\), then there is a finite subset \(T\) of \(G'\) with
\(G' \subseteq \gen{T \cup\{g\}}\).
\end{claim}

\begin{proof}
Notice that there is an
\(\epsilon>0\) so that \(g\) moves all points in \((0,\epsilon)\)
and \((1-\epsilon,1)\).  It follows that for some \(x<y\) in
\(\Z[1/2]\) the images of \((x,y)\) under integral powers of \(g\)
cover \((0,1)\).
Hence the conjugates of \(F_{[x,y]}\) under integral powers of \(g\) generate \(F'=G'\).
The claim now follows from the fact that
\(F_{[x,y]} \cong F\) is generated by two elements.
\end{proof}

In what follows, we will refer to the closure of the support of an \(f\in\HomeoI\)
as the \emph{extended support of \(f\)}.

\begin{claim} \label{EquivCommute}
Suppose \(g \in G \setminus G'\) has connected extended support and \(h \in G\).
If \(g\) commutes with \(g^h\) then \(g\) commutes with \(h\).
\end{claim}

\begin{proof}
We know that 0 or 1 is in the extended support of \(g\) and both 0 and 1 are fixed by \(h\). 
Since \(g\) has only finitely many bumps, it has finitely many transition points.
If any of these are moved by \(h\), then at least one fixed point of
\(g^h\) is moved by \(g\) or at least one fixed point of \(g\) is
moved by \(g^h\).
In either case this implies that \(g\) does not commute with \(g^h\).
If no transition point of \(g\) is moved by \(h\), then the orbitals of \(g^h\) are
those of \(g\).
It follows from the chain rule that on each
orbital \(J\) of \(g\), \(g\) and \(g^h\) agree on a small neighborhood of the
endpoints of \(J\).
By 
\cite[\S 4]{MR1855112} two elements of \(\PLoI\) commute over a
common orbital of support only if they admit a common root over that orbital.
In particular, if \(g \restriction J\) and \(g^h \restriction J\)
agree in a neighborhood of the endpoints of \(J\)
and they commute, then they must be equal.  Hence \(h\) commutes with \(g\) over each orbital of support of \(g\), so \(h\) and \(g\) commute.
\end{proof}

We now turn our attention to our representation of \(G\) as
a subgroup \(H=\gen{X}\) of \(\HomeoI\) where \(X\) is geometrically proper.
If \(H\) has more than one component of support, then the
restriction to each is a quotient of \(H\).
Since no non-trivial element of \(F\) commutes with every element of \(F'\) and since \(F'=G'\) is simple,
it follows that every
nontrivial normal subgroup of \(G\) contains \(G'\).
Consequently every proper quotient of \(H\cong G\) is abelian.
Thus if no restriction were an isomorphism, \(H\) would be abelian (which is absurd).
We can thus replace \(H\) by its restriction to a component of its support on which the
restriction is faithful, and further, 
we can conjugate by a homeomorphism of \(\R\)
so that the support is \((0,1)\).  Note that this new embedding of \(G\) in \(\HomeoI\) is geometrically proper if the original embedding is geometrically proper.

Let \(\pi_0\) denote the restriction of the germ homomorphism \(\gamma_0^+\) to \(H\) and 
\(\pi_1\) denote the restriction of \(\gamma_1^-\) to \(H\);
define \(\pi_H(h) := (\pi_0(h),\pi_1(h))\).
Observe that since \(X\) is geometrically proper we have that
for each \(i < 2\), there is at most one
element of \(X\) which is not in the kernel of \(\pi_i\).
Observe that this implies the image of \(\pi_H\) is abelian and hence
\(H' \subseteq \ker (\pi_H)\).

\begin{claim} \label{NonCyclicImage}
The image of \(\pi_H\) is not cyclic.
\end{claim}

\begin{proof}
Recall that \(G' = \ker (\pi_G)\) and \(H' \subseteq \ker(\pi_H)\).
In particular, \(\phi\) induces a well defined homomorphism
from \(G/G'\) to \(H/\ker (\pi_H)\).
If the image of \(\pi_H\) were cyclic, then this homomorphism would
have a nontrivial kernel. 

Pick \(g:=f_0^m f_1^n \in G\setminus G'\) in the kernel of \(\pi_H\).
The element \(g\) will have connected extended support.
As \(\phi(g)\in\ker(\pi_H)\) we must have the support of \(\phi(g)\) is contained in \([x,y]\) for
some \(0 < x < y < 1\).  As the support of \(H\) is \((0,1)\) there is an \(h \in H\) so that \(xh>y\), and for this \(h\) we have 
\(\phi(g)\) commutes with \(\phi(g)^h\), but not with \(h\).
Now, \(g\) commutes with \(\phi^{-1} (\phi(g)^h)\), but not with \(\phi^{-1}(h)\), however, \(\phi^{-1}(\phi(g)^h) = g^{\phi^{-1}(h)}\),
contradicting Claim \ref{EquivCommute}.
\end{proof}

At this point we know that, by the geometric properness of \(X\), for each coordinate \(i\in\{0,1\}\)
there is a unique element \(h_i\in X\) so that
\(i\) is in the extended support of \(h_i\) and so that \(h_0\ne h_1\).
The fact that \(h_0\) and \(h_1\) are distinct and unique implies that
the image of \(\pi_H\) is the product of the images of \(\pi_0\) and \(\pi_1\).

\begin{claim}\label{claim_noFiniteGenSet}
For each element \(h\) of \(X\) whose extended support contains 0 or 1,
there is no finite subset \(T\) of the kernel of \(\pi_{H}\) so that
\(\ker(\pi_{H}) \subseteq \gen{T\cup \{h\}}\).
\end{claim}

\begin{proof}
Let \(h\) be given and suppose without loss of generality that \(0\) is
in the extended support of \(h\).
As noted above, \(1\) can not be in the extended support of \(h\). 
Since \(\gen{X}'\subseteq \ker(\pi_{H})\),
there are nontrivial elements in \(\ker(\pi_{H})\).
Because the orbital of \(H\) is \((0,1)\),
there are points arbitrarily close to 0 and 1 moved by elements of
\(\ker (\pi_{H})\).
It follows that if \(T\subseteq
\ker(\pi_{H})\) is finite, then there is a neighborhood of \(1\)
fixed by all elements of \(\gen{T\cup \{h\}}\) and thus
\(\gen{T\cup \{h\}}\) cannot contain all of \(\ker(\pi_{H})\).
\end{proof}

In order to finish the proof, it suffices to show that \(\pi_G\left(\phi^{-1}(h_0)\right)\) and
\(\pi_G\left(\phi^{-1}(h_1)\right)\) generate the image of \(\pi_G\) and that
each have exactly one (necessarily different) nonzero coordinate. 
Since, by the proof of Claim \ref{NonCyclicImage},
\(\phi\) induces an isomorphism \(G/G' \cong H/\ker (\pi_H)\), it follows that
\(\{\pi_G(\phi^{-1}(h_i)) \mid i < 2\}\) must generate \(G/G'\).
On the other hand, by Claims \ref{BothGermsNontrivial} and \ref{claim_noFiniteGenSet},
it must be that one coordinate
of \(\pi_G(\phi^{-1}(h_i))\) must be 0 for each \(i \in \{0,1\}\).
This shows that these two elements, which generate the group \(K\), together make a set of the form \(\left\{(p,0),(0,q)\right\}\) for some non-zero \(p\) and \(q\), and therefore  \(G \cong F\).
\end{proof}

\begin{remark} 
The group \(E=\{(p,q)\in\Z\times\Z\mid p+q \equiv0 \mod2\}\)
is a subgroup of \(\Z \times \Z\) which is not of the form \(P \times Q\) and hence
\(\pi^{-1}_F (E)\) is a finite index subgroup of \(F\) which is not isomorphic to \(F\).
In particular, there are finite index subgroups of \(F\) which do not admit
geometrically proper generating sets.
\end{remark}

\section{Abstract Ping-Pong Systems}

\label{AbstractPingPongSec}

In this section we will abstract the analysis of geometrically
fast systems of bumps in previous sections to
the setting of permutations of a set \(S\).
(By \emph{permutation} of \(S\) we simply mean a bijection from \(S\) to \(S\).)
Our goal will be to state the analog of Theorem \ref{CombToIso} and its consequences.
The proofs are an exercise for the reader.

Suppose now that \(A\) is a collection of permutations of a set \(S\) such that
\(A \cap A^{-1} = \emptyset\).
A \emph{ping-pong system} on \(A\) is an assignment \(a \mapsto \dest(a)\) of sets
to each element of \(A^\pm\) such that whenever \(a\) and \(b\) are in \(A^\pm\) and \(s \in S\):
\begin{itemize}

\item \label{basic_Dp}
\(\dest(a) \subseteq \supt(a)\) and if \(s \in \supt(a)\), then
there is an integer \(k\) such that \(s a^k \in \dest (a)\);

\item \label{ping-pong_cond}
if \(s \in \supt(a)\), then \(s a  \in \dest(a)\) if and only if \(s \not \in \src(a) : = \dest(a^{-1})\);

\item if \(a \ne b\), then \(\dest(a) \cap \dest(b)\) is empty;

\item if \(\dest(a) \cap \supt(b) \ne \emptyset\), then \(\dest(a) \subseteq \supt(b)\).

\end{itemize}
\noindent
The following lemma summarizes some immediate consequences of this definition.
\begin{lemma}
Given a set \(S\) and  a collection \(A\) of permutations of \(S\)
equipped with a ping pong system, the following are true:
\begin{itemize}

\item if \(a \in A\) and \(s \in \supt (a)\), then there is a unique \(k \in \Z\) such that:
\[
s a^k \in \supt(a) \setminus (\src(a) \cup \dest(a))\]

\item if \(a \in A\), then \(\dest(a) a \subseteq \dest(a)\);

\end{itemize}
In particular, all elements of \(A\) have infinite order.
\end{lemma}
As remarked in Section \ref{FastBumpsSec}, geometrically fast sets of bumps admit
a ping-pong system.
The meanings of \emph{source}, \emph{destination}, and \emph{locally reduced word}
all readily adapt to this new context.
Furthermore, the proofs of Lemmas \ref{LocRedBasics}--\ref{FellowTraveler3} given in Section \ref{PingPongSec}
use only the axiomatic properties of a ping-pong system and thus these lemmas 
are valid in the present context.
The next example is simplistic, but it will serve to illustrate a number of points in this section.

\begin{example} \label{PSL2Example}
View the real projective line \(\mathbf{P}\) as \(\R \cup \{\infty\}\) and \(\PSL_2(\Z)\) as a group of
fractional linear transformations of \(\mathbf{P}\).
The homeomorphisms \(\alpha\) and \(\beta\) of \(\mathbf{P}\) defined by
\[
t\alpha := t + 1 \qquad t \beta := \frac{t}{1-t}
\]
generate \(\PSL_2(\Z)\).
If we take \(A = \{\alpha^2,\beta^2\}\), then
\[
\src(\alpha^2) := (-\infty,-1)  \qquad \qquad \dest(\alpha^2) := [1,\infty)
\]
\[
\src(\beta^2) := (0,1) \qquad \qquad \dest(\beta^2) := [-1,0) 
\]
defines a ping-pong system.
It is well known that \(\gen{\alpha^2,\beta^2}\) is free;
in fact this is one of the classical applications of the Ping-Pong Lemma.
\end{example}

In order to better understand \(\gen{A}\) when \(A\) is a set of permutations admitting a
ping-pong system, it will be helpful to represent \(\gen{A}\) as a family of homeomorphisms
of a certain space \(K_A\).
This compact space can be thought of as a space of \emph{histories} in the sense
of Section \ref{PingPongSec}.  
If \(S\) is the underlying set which elements of \(A\) permute,
let \(M = M_A\) denote the collection of all sets of the form
\[
\tilde s: = \{a \in A \mid s \in \supt(a)\}
\]
where \(s \in S \setminus \bigcup \{\dest(a) \mid a \in A^\pm\}\).
Elements of \(M\) will play the same role as the 
initial markers of a geometrically fast collection of bumps.

\begin{example} \label{PSL2Marker}
Continuing with the Example \ref{PSL2Example}, \(M\) consists of two points:
\(\widetilde 0 = \{\alpha\}\) and \(\widetilde \infty = \{\beta\}\).
We can also restrict the action of \(\PSL_2(\Z)\) on \(\mathbf{P}\) to the irrationals.
In this case \(M\) is empty.
\end{example}

It will be convenient to define \(\dest(\tilde s) := \bigcap \{\supt(a) \setminus \src(a) \mid a \in \tilde s\}\) and
\(\supt(\tilde s) := \emptyset\).
Define \(K_A\) to be all \(\eta\) such that:
\begin{itemize}

\item \(\eta\) is a suffix closed family of finite 
strings in the alphabet \(A^\pm \cup M\);

\item if \(\str{ab}\) are consecutive symbols of an element of \(\eta\), then
\(\dest(a) \subseteq \supt(b) \setminus \src(b)\);

\item for each \(n\), there is at most one element of \(\eta\) of length \(n\);

\item if \(\str{w} \in \eta\) and \(\eta\) does not contain a symbol from \(M\),
then \(\str{w}\) is a proper suffix of an element of \(\eta\).

\end{itemize}
\noindent
The second condition implies that elements of \(\eta\) are freely reduced since if \(b = a^{-1}\),
then \(\src(b) = \dest(a)\).
Observe that if \(\str{w}\) is in \(\eta\), the only occurrence of an element of \(M\) in \(\str{w}\) must
be as the first symbol of \(\str{w}\) (and there need not be any occurrence of an element of \(M\) in \(\str{w}\)).

Notice that every \(\eta \in K_A\) has at least one element other than \(\varepsilon\) and that all elements
of \(\eta\) of positive length must have the same final symbol.
We define \(\dest(\eta) := \dest(a)\) where \(\str{a}\) is the final symbol of every
element of \(\eta\) other than \(\varepsilon\).
We topologize \(K_A\) by declaring that
\([\str{w}] : = \{\eta \in K_A \mid \str{w} \in \eta\}\) is closed and open.
Notice that if \(\eta\) is finite, it is an isolated point of \(K_A\).

\begin{prop}
\(K_A\) is a Hausdorff space and if \(A\) is finite, then \(K_A\) is compact.
\end{prop}

\noindent
Each \(a \in A^\pm\) defines a homeomorphism \(\hat a : K_A \to K_A\) by:
\[
\eta \hat a := 
\begin{cases}
\{\str{ua} \mid \str{u} \in \eta\} \cup \{\varepsilon\} & \textrm{ if } \dest(\eta) \subseteq \supt(a) \setminus \src(a) \\
\{\str{u} \mid \str{ua}^{-1} \in \eta\}  & \textrm{ if } \dest (\eta) = \src(\str{a}) \\
\eta & \textrm{ if } \dest(\eta) \cap \supt(\str{a}) = \emptyset
\end{cases}
\]
Thus \(\eta \hat a\) is obtained by 
appending 
\(\str{a}\) to the end of every element of \(\eta\),
collecting the local reductions, and possibly including \(\varepsilon\).
Set \(\hat A = \{\hat a: a \in A\}\).

We say that a ping-pong system on \(A\) is \emph{faithful} if
\(\Lambda_A : = \{\eta \in K_A : \eta \textrm{ is finite}\}\) is dense in \(K_A\) (i.e.
whenever \(\str{w}\) is in some \(\eta \in K_A\), there is a finite \(\eta' \in K_A\)
which has \(\str{w}\) as an element).

\begin{example} \label{NonfaithfulExample}
As noted above, if we restrict the elements of \(\PSL_2(\Z)\) to the irrationals,
then  \(M = \emptyset\) and in particular the system is not faithful.
On the other hand, 
\[
\dest(\alpha^4) := (2,\infty) \cap S \qquad \qquad \dest(\alpha^{-4}) := (-\infty,-2) \cap S
\]
\[
\dest(\beta^4) := (-1/2,0) \cap S \qquad \qquad \dest(\beta^{-4}) := (0,1/2) \cap S.
\]
defines a ping-pong system in which \(M\) contains a single element \(\{\alpha,\beta\}\).
\end{example}

\noindent
While not every ping-pong system is faithful,
the reader is invited to verify that if \(A\) admits a ping-pong system, then
\(\{a^2 \mid a \in A\}\) admits a faithful ping-pong system.

\begin{thm}
If \(A\) is a set of permutations which admits a ping-pong system, then
\(a \mapsto \hat a\) extends to an epimorphism of \(\gen{A}\) onto \(\gen{\hat A}\).
If the ping-pong system is faithful, then the epimorphism is an isomorphism.
\end{thm}

The map \(x \mapsto \eta(x)\) defined in Section \ref{PingPongSec} adapts \emph{mutatis mutandis}
to define a map \(s \mapsto \eta(s)\) from \(S\) into \(K_A\).
It is readily verified that if \(a \in A\) and \(s \in S\), then \(\eta(sa) = \eta(s)\hat a\).
The existence of a faithful ping-pong system also has the following structural consequence
which follows readily from the abstract form of Lemma \ref{TreeAction}.

\begin{prop}
If \(A\) admits a faithful ping-pong system, then \(\gen{A}\) is torsion free.
\end{prop}
\noindent
This shows in particular that \(\PSL_2(\Z) = \gen{\alpha,\beta}\) --- which contains elements of finite order such as
\(t \mapsto -1/t\) --- does not admit a ping-pong system.

A \emph{blueprint} for a ping-pong system is a pair \(\Bfrak = (\str{B},\str{supt})\)
such that:
\begin{itemize}

\item \(\str{B}\) is a set and \(\str{supt}\) is a binary relation on \(\str{B}\)
which is interpreted as a set-valued function:
\(\str{b} \in \str{supt}(\str{a})\) if \((\str{a}, \str{b}) \in \str{supt}\);

\item if \(\str{a} \in \str{B}\) and \(\str{supt}( \str{a})\) is nonempty, then \(\str{a} \in \str{supt}(\str{a})\);

\item if \(\str{a} \in \str{A}\), then there is a unique \(\str{a}^{-1} \in \str{A} \setminus \{\str{a}\}\) with
\(\str{supt} (\str{a}^{-1}) = \str{supt} (\str{a})\).

\end{itemize}
Additionally, setting 
\(\str{A} : = \{\str{a} \in \str{B} \mid \str{supt}(\str{a}) \ne \emptyset\}\) we require that:
\begin{itemize}

\item if \(\str{b} \ne \str{c} \in \str{B} \setminus \str{A}\), then \(\str{\tilde b} \ne \str{\tilde c}\) where
\({\str{\tilde b}} := \{\str{a} \in \str{A} \mid \str{b} \in \str{supt}(\str{a}) \}\).
 
\end{itemize}
Two blueprints are isomorphic if they are isomorphic as structures.
If \(A\) is a set of permutations which admits a ping-pong system, then the blueprint
\(\Bfrak_A = (\str{B}_A,\str{supt}_A)\)
for the system is defined by
\(\str{B}_A := \{\str{a} \mid a \in A^\pm\} \cup \{{\str{\tilde  s}} \mid \tilde s \in \widetilde S\}\)
with \(\str{b} \in \str{supt}_A(\str{a})\) if \(\dest(b) \subseteq \supt(a)\).
Also, if \(\Bfrak = (\str{A},\str{supt})\) is a blueprint for a ping-pong system, then
one defines \(K_{\Bfrak}\) and homeomorphisms \({\str{\hat a}} : K_{\Bfrak} \to K_{\Bfrak}\)
for \(\str{a} \in \str{A}^{\Bfrak}\) by a routine adaptation of the construction above.
In fact \(K_A\) is can be regarded as factoring though its blueprint in the sense that
\(K_A = K_{\Bfrak_A}\) modulo identifying \(a\) and \(\str{a}\).
It is routine to check that the following theorem holds as well.

\begin{thm} \label{AbsCombToIso}
If \(\Bfrak_0\) and \(\Bfrak_1\) are isomorphic blueprints for ping-pong systems,
then the isomorphism induces an homeomorphism \(\theta:K_{\Bfrak_0} \to K_{\Bfrak_1}\)
such that \(g \mapsto g^\theta\) defines an isomorphism between
\(\gen{\hat {\str{a}} \mid \str{a} \in \str{A}_0}\) and \(\gen{\hat {\str{a}} \mid \str{a} \in \str{A}_1}\).
\end{thm}

If \(A\) is a set of permutations which admits a faithful ping-pong system,
then the blueprint for \(A\) and for \(\{a^{k(a)} \mid a \in A\}\) are canonically
isomorphic whenever \(a \mapsto k(a)\) is an assignment of a positive integer to each element of \(A\).

\begin{cor}
Any set of permutations \(A\) which admits a faithful ping-pong system is an algebraically
fast generating set for \(\gen{A}\).
\end{cor}

If the system is not faithful, then \(\{a^2 \mid a \in A^2\}\) may contain new markers, as
was illustrated in Example \ref{NonfaithfulExample}.
Notice that in this example,
both \(\gen{\alpha^2,\beta^2}\) and \(\gen{\alpha^4,\beta^4}\) are free --- and hence marked isomorphic ---
even though the blueprints associated to the ping-pong systems are not isomorphic.

A blueprint \(\Bfrak\) is (cyclically) orderable if there is a (cyclic) ordering on \(\str{B}\)
such that for all \(\str{a} \in \str{A}\), 
\(\str{supt}(\str{a})\) is an interval in the (cyclic) ordering with endpoints \(\src(\str{a})\) and \(\dest(\str{a})\).
It is readily checked that the (cyclic) order on \(\str{B}\)
induces a reverse lexicographic (cyclic) order on \(K_{\Bfrak}\) which is preserved by the
homeomorphisms \(\hat {\str{a}}\) for \(\str{a} \in \str{A}\).

\begin{cor}
If \(A\) is a finite set of permutations which admits a faithful ping-pong system, then
\(\gen{A}\) embeds into Thompson's group \(V\).
If the blueprint of \(A\) is cyclically orderable, then \(\gen{A}\) embeds into
\(T\).
If the blueprint of \(A\) is orderable, then \(\gen{A}\) embeds into \(F\). 
\end{cor}

\begin{proof}
It is readily verified that any finite blueprint can be realized as the blueprint of
a ping-pong system of a finite subset of \(V\).
Moreover, if the blueprint is cyclically orderable (or orderable), then it can be realized
by elements of \(T\) (respectively \(F\)). 
The corollary follows by Theorem \ref{AbsCombToIso}.
\end{proof}

\providecommand{\bysame}{\leavevmode\hbox to3em{\hrulefill}\thinspace}
\providecommand{\MR}{\relax\ifhmode\unskip\space\fi MR }

\providecommand{\MRhref}[2]{
  \href{http://www.ams.org/mathscinet-getitem?mr=#1}{#2}
}
\providecommand{\href}[2]{#2}


\begin{thebibliography}{10}

\bibitem{MR2383051}
C.~Bleak, \emph{An algebraic classification of some solvable groups of
  homeomorphisms}, J. Algebra \textbf{319} (2008), no.~4, 1368--1397.

\bibitem{MR2466019}
\bysame, \emph{A minimal non-solvable group of homeomorphisms}, Groups Geom.
  Dyn. \textbf{3} (2009), no.~1, 1--37.

\bibitem{bleak+wassink}
C.~Bleak and B.~Wassink, \emph{Finite index subgroups of {R}.
  {T}hompson{'}s group {F}}, ArXiv preprint:
  https://arxiv.org/pdf/0711.1014.pdf, 2007.

\bibitem{brin:ubiq}
Matthew~G. Brin, \emph{The ubiquity of {T}hompson's group \({F}\) in groups of
  piecewise linear homeomorphisms of the unit interval}, J. London Math. Soc.
  (2) \textbf{60} (1999), no.~2, 449--460. 

\bibitem{MR2160570}
\bysame, \emph{Elementary amenable subgroups of {R}. {T}hompson's group \({F}\)},
  Internat. J. Algebra Comput. \textbf{15} (2005), no.~4, 619--642.


\bibitem{brin+fer}
Matthew~G. Brin and Fernando Guzm{\'a}n, \emph{Automorphisms of generalized
  {T}hompson groups}, J. Algebra \textbf{203} (1998), no.~1, 285--348.

\bibitem{brin+squier}
M.~G. Brin and C.~C. Squier, \emph{Groups of piecewise linear
  homeomorphisms of the real line}, Invent. Math. \textbf{79} (1985), 485--498.

\bibitem{MR1855112}
\bysame, \emph{Presentations, conjugacy, roots, and
  centralizers in groups of piecewise linear homeomorphisms of the real line},
  Comm. Algebra \textbf{29} (2001), no.~10, 4557--4596. 

\bibitem{brown:finiteprop}
K.~S. Brown, \emph{Finiteness properties of groups}, Journal of Pure and
  Applied Algebra \textbf{44} (1987), 45--75.

\bibitem{CFP}
J.~W. Cannon, W.~J. Floyd, and W.~R. Parry, \emph{Introductory notes on
  {R}ichard {T}hompson's groups}, Enseign. Math. (2) \textbf{42} (1996),
  no.~3-4, 215--256. 

\bibitem{Fricke+Klein}
R.~Fricke and F.~Klein, \emph{Vorlesungen \"uber die theorie der automorphen
  funktionen}, vol.~1, Teubner, Leipzig, 1897.

\bibitem{kim+koberda+lodha}
S.~Kim, T.~Koberda, and Y.~Lodha, \emph{Chain groups of
  homeomorphisms of the interval and the circle}, ArXiv preprint:
  https://arxiv.org/pdf/1610.04099.pdf, 2016.

\bibitem{Klein:Riemann}
F.~Klein, \emph{Neue {B}eitr\"age zur {R}iemann'schen {F}unctionentheorie},
  Math. Annalen \textbf{21} (1883), 141--218.

\bibitem{MR0148731}
A.~M. Macbeath, \emph{Packings, free products and residually finite groups},
  Proc. Cambridge Philos. Soc. \textbf{59} (1963), 555--558. 

\bibitem{MR2135961}
A.~Navas, \emph{Quelques groupes moyennables de diff\'eomorphismes de
  l'intervalle}, Bol. Soc. Mat. Mexicana (3) \textbf{10} (2004), no.~2,
  219--244 (2005).

\bibitem{free_subgrp_linear}
J.~Tits, \emph{Free subgroups in linear groups}, J. Algebra \textbf{20} (1972), 250--270.

\bibitem{x+1_x^p_free}
S.~White, \emph{The group generated by \(x \mapsto x+1\) and \(x \mapsto x^p\)
  is free}, J. Algebra \textbf{118} (1988), 408--422.

\end{thebibliography}
\end{document}